\documentclass[11pt,a4paper,leqno]{amsart}

\usepackage[latin1]{inputenc}
\usepackage[T1]{fontenc}
\usepackage{amsfonts}
\usepackage{amsmath}
\usepackage{amssymb}
\usepackage{eurosym, enumerate}
\usepackage{mathrsfs}
\usepackage{palatino}
\usepackage{color}
\usepackage{xcolor}
\usepackage{esint}
\usepackage{url}
\usepackage{graphicx}

\numberwithin{equation}{section}

\newcommand{\ud}[0]{\,\mathrm{d}}
\newcommand{\floor}[1]{\lfloor #1 \rfloor}

\newcommand{\ave}[1]{\langle #1\rangle}

\renewcommand{\Re}[0]{\operatorname{Re}}

\newcommand{\Rd}{\mathbb{R}^d}
\newcommand{\f}{\frac}
\newcommand{\lesi}{\lesssim}

\theoremstyle{plain}
\newtheorem{thm}{Theorem}[section]
\newtheorem{lem}[thm]{Lemma}
\newtheorem{prop}[thm]{Proposition}
\newtheorem{cor}[thm]{Corollary}

\theoremstyle{definition}
\newtheorem{defn}[thm]{Definition}

\theoremstyle{remark}
\newtheorem{rem}[thm]{Remark}

\title[fractional operators beyond Calder\'on--Zygmund theory]{New Sparse Domination and Weighted Estimates for Fractional Operators Beyond Calder\'on-Zygmund Theory}

\author{The Anh Bui}

\author{Linfei Zheng}

\address{School of Mathematical and Physical Sciences, Macquarie University, NSW 2109, Australia}
\email{the.bui@mq.edu.au}
\address{Center for Applied Mathematics, Tianjin University, Weijin Road 92, 300072 Tianjin, China}
\email{linfei\_zheng@tju.edu.cn}
\makeatletter
\@namedef{subjclassname@2020}{%
  \textup{2020} Mathematics Subject Classification}
\makeatother


\begin{document}

\allowdisplaybreaks

\begin{abstract}
Let $L$ be a closed, densely defined operator on $L^2(\mathbb{R}^n)$ satisfying suitable $L^p-L^q$ off-diagonal estimates of order $\kappa > 0$. This paper aims to investigate the two-weight estimate and the Bloom weighted estimate for the fractional operator $L^{-\alpha/\kappa}$ with $0 < \alpha < n$ through the method of sparse domination. Our assumptions on the operators are minimal, and our result applies to a wide range of differential operators. As a byproduct, we also establish a new sparse domination criterion for a general class of fractional operators, including the classical fractional integral.
\end{abstract}

\maketitle

\section{Introduction and main results}

\subsection{Statement of results}
\par Over the past few decades, there has been a great focus in harmonic analysis on addressing weighted inequalities concerning classical operators such as the Hilbert and Riesz transforms, along with other singular integral operators. A particularly challenging problem has been understanding the precise dependence of norm estimates for a given operator in relation to the "norm" of the corresponding weights. The first result is due to S. Buckley \cite{Bu} where the sharp weighted estimate for the Hardy--Littlewood operator $M$ was given:
\begin{equation}\label{Buckley_thm}
	\|M\|_{L^p(w)} \leq C_{p} \,[w]_{A_p}^{\frac{1}{p-1}}, \ \ 1<p<\infty,
\end{equation}
where 
\begin{equation*}
	[w]_{A_p} := \sup_{Q: \,\text{cube in}\, \mathbb{R}^n} \left(\frac{1}{|Q|}\int_Q w(x) \ud x\right) \left(\frac{1}{|Q|}\int_Q w(x)^{-\frac{1}{p-1}}\ud x\right)^{p-1} < \infty, \hspace{1em} p>1,
\end{equation*}
and $C_p$ is a dimensional constant that also depends on $p$, but not on $w$. We say that the estimate in  \eqref{Buckley_thm} is sharp in the sense that the exponent $1/(p-1)$ cannot be replaced by a smaller one.

The sharp weighted estimates for singular integral operators are  much more difficult. The sharp  bounds for the Hilbert and Riesz transforms were investigated  by Petermichl \cite{Pet,Pet2}, while the long standing $A_2$ conjecture or the sharp weighted estimates for the Calder\'on--Zygmund operators were addressed in \cite{Hyt1}. Subsequently, research in this area progressed rapidly, with developments in sharp weighted estimates and related problems. For a more comprehensive background, we refer to \cite{Hyt1, Hyt2, HLP, L, LN} and the references therein.

An effective method in weighted estimates is to use sparse operators to dominate singular integral operators. This idea was originally proposed by Lerner \cite{L4}, in that paper, he obtained the following estimate: 
\begin{equation}{\label{Lerner's sparse domination}}
\|{T_{\#}}f\|_{X} \leqslant c(n,\operatorname{T})\sup_{\mathscr{D},\mathcal{S}}\|\mathcal{A}_{\mathscr{D},\mathcal{S}}|f|\|_{X},    
\end{equation}
where $X$ is a Bananch space, $T_{\#}$ is the maximal truncated Calder\'on--Zygmund operator, and $\mathcal{A}_{\mathscr{D},\mathcal{S}}$ is the sparse operator defined by 
$$\mathcal{A}_{\mathscr{D},\mathcal{S}}f:=\sum\limits_{Q \in \mathcal{S}}\ave{f}_{Q}\chi_{Q}.$$ 
In \cite{L}, Lerner provided a new proof for estimate (\ref{Lerner's sparse domination}) and obtained an alternative proof for $A_2$ conjecture. 
Later on, the method of sparse domination received a lot of attention, many classical operators have been proven to be dominated by sparse operators. For example, Lacey \cite{Lacey1} established a pointwise sparse domination for Calder\'on--Zygmund operators satisfying Dini-type condition.

In addition to weighted estimates, sparse domination method also has many other applications, one of which is to prove the Bloom weighted boundedness for commutators. The study of Bloom weighted estimate for commutators originated from Bloom \cite{BL}, where he obtained a characterization of weighted BMO space through the two-weight boundedness for commutators of Hilbert transform $[b,\operatorname{H}]f:=b\operatorname{H}(f)-\operatorname{H}(bf).$ Specifically, it was proved that
\begin{equation}{\label{Bloom's result}}
\|b\|_{{BMO}_{\nu}} \lesssim \|[b,\operatorname{H}]\|_{L^{p}(\mu) \rightarrow L^{p}(\lambda)} \lesssim \|b\|_{{BMO}_{\nu}},
\end{equation}
where $\mu,\lambda \in A_{p}, \nu=\mu^{\frac{1}{p}}\lambda^{-\frac{1}{p}},$ and
$$\|b\|_{{BMO}_{\nu}}:=\sup_{Q: \,\text{cube in}\, \mathbb{R}^n}\frac{1}{\nu(Q)}\int_Q|b-\ave{b}_Q| .$$

In recent years, the upper bound estimate in (\ref{Bloom's result}) has been extended to general Calder\'on--Zygmund operators, as shown in \cite{ST, HLB, LOR, HOJ}. We point out that the main strategies used in \cite{LOR, HOJ} are exactly the sparse domination mentioned above. Once the Bloom weighted estimate for commutators is obtained, a natural question is whether such estimate also holds for the $k$-iterative commutators of Calder\'on--Zygmund operators $ {T}$, which refers to   
$${T}_{b}^{m}f=[b, {T}_{b}^{m-1}]f,\qquad {T}_{b}^{1}f=[b,T]f.$$
This problem was initially studied under the assumption of $b \in BMO(\mathbb{R}^n) \cap BMO_{\nu}(\mathbb{R}^n)$ in \cite{HW,Hyt3}, and later \cite{LOR2} relax it to $b \in BMO_{\nu^{1/k}}(\mathbb{R}^n).$ In a recent work \cite{LLO}, Lerner, Lorist and Ombrosi provided quantitative Bloom weighted estimates for a class of bilinear sparse forms related to iterative commutators, futhermore, they obtained Bloom weighted estimates for iterative commutators of operators under certain diagonal boundedness assumptions.

This paper aims to establish quantitative two-weight estimates for a generalized fractional integral and Bloom weighted estimates for its iterative commutators through sparse domination, motivated by the fractional powers of abstract operators extending beyond the Calder\'on--Zygmund theory. Before coming to the details, we provide a review of some classical results. Let $0 < \alpha < n$, the fractional integral operator $I_\alpha$ on $\mathbb{R}^n$ is defined by
\begin{equation}\label{eq-fractional integral}
	I_\alpha f(x) := \int_{\mathbb{R}^n} \frac{f(y)}{|x - y|^{n-\alpha}} \, \ud y.
\end{equation}
It is well known that the fractional operator $I_\alpha$ is bounded from $L^p(\mathbb R^n)$ into $L^q(\mathbb R^n)$ provided that $1 < p < \frac{n}{\alpha}$ and
\[
\frac{1}{q} = \frac{1}{p} - \frac{\alpha}{n}.
\]

Regarding the  weighted estimates for the fractional integral $I_\alpha$, Muckenhoupt and Wheeden \cite{MW} introduced the  a new class of weights denoted by  $A_{p,q}$ for $1 < p < q < \infty$.  Recall that for $1 < p < q < \infty$, the weight $w$, which  is a non-negative locally integrable function, is in the class $A_{p,q}$ if
\[
[w]_{A_{p,q}} := \sup_Q \left( \frac{1}{|Q|} \int_Q w^q \right) \left( \frac{1}{|Q|} \int_Q w^{-p'} \right)^{q/p'} < \infty,
\]
where the supremum is taken over all balls/cubes $Q$ in $\mathbb R^n$. It is important to note that if $w \in A_{p,q}$ then $w^q \in A_{1+q/p'}$ with $[w^q]_{1+q/p'} = [w]_{A_{p,q}}$ and $w^{-p'} \in A_{1+p'/q}$ with $[w^{-p'}]_{1+p'/q} = [w]_{A_{p,q}}^{p'/q}$ where $A_s$ denotes the classical Muckenhoupt class of weights. Then it was proved in \cite{MW} that for $0 < \alpha < n$, $1 < p < \frac{n}{\alpha}$, and $\frac{1}{q} = \frac{1}{p} - \frac{\alpha}{n}$, we have
\[
\|I_\alpha f\|_{L^q(w^q)} \leq C(n,\alpha,[w]_{A_{p,q}}) \|f\|_{L^p(w^p)}, \ \ \ \ w \in A_{p,q},
\]
where the constant $C$ depends on $n, \alpha$ and $[w]_{A_{p,q}}$. Note that this result does not reflect the quantitative dependence of the  operator norm in terms of the relevant constant involving $[w]_{A_{p,q}}$. The sharp $A_{p,q}$ weighted estimate  was addressed by Lacey et. al. \cite{Lacey}. It was proved in \cite{Lacey} that for $0 < \alpha < n$, $1 < p < \frac{n}{\alpha}$, and $\frac{1}{q} = \frac{1}{p} - \frac{\alpha}{n}$ we have
\begin{equation}{\label{Lacey's estimate}}
   \|I_\alpha f\|_{L^q(w^q)} \leq C_{n,\alpha}[w]_{A_{p,q}}^{(1-\frac{\alpha}{n})\max\{1,\frac{p'}{q}\}}  \|f\|_{L^p(w^p)}, \ \ \ w\in A_{p,q} 
\end{equation}
and the estimate is sharp in the sense that the inequality does not hold if we replace the exponent of the $A_{p,q}$ constant by a smaller one. The result has been extended to a class of fractional operator with $L^{\alpha, r'}$-H\"ormander condition, as detailed in \cite{IRV}. This condition can be viewed as an $L^{r'}$-H\"ormander condition for the Calder\'on-Zygmund operators. Note that the approaches in \cite{Lacey} and \cite{IRV} rely heavily on the pointwise estimates of the kernel of the fractional integral operators.

In \cite{AM1}, the weighted norm inequalities for fractional powers of second order elliptic operator in divergence form were investigated. It is interesting to emphasize that the fractional operators considered in \cite{AM1} can be viewed as non-integral operators as we do not know any information on the pointwise estimates of the heat kernel of the second order elliptic operator in the general case and hence these classes of fractional operators beyond the Calder\'on--Zygmund theory. To the best of our knowledge, the sharp weighted estimates and the sparse domination for such operators have not been studied yet. Motivated by these research papers, in this paper we aim to conduct further research on the weighted estimates of a class of fractional operators beyond the Calder\'on-Zygmund theory. 

Before coming to the details, we would like to set up our framework for commutators systematically. Given a linear operator $T$ and $b \in L_{\text{loc}}^{1}(\mathbb{R}^n)$, we define $$T_{b}^{1}(f):=bT(f)-T(bf)$$ as the first order commutator. For $m \in \mathbb{N}, m \geqslant 2$, we inductively define $$T_{b}^{m}(f):=bT_{b}^{m-1}(f)-T_{b}^{m-1}(bf)$$ as the higher order commutator. It is easy to verify that 
\begin{equation}\label{commutator}
	T_{b}^{m}(f)(x)=T((b(x)-b(\cdot))^mf)(x),\quad x\in \mathbb{R}^n,\,m \geqslant 1.
\end{equation}
Suppose $T$ is a general operator, we use formula (\ref{commutator}) as the definition of iterative commutators. Moreover, we define $T_{b}^{0}(f)(x):=T(f)(x)$.

In this section, we will consider a class of operators satisfying the assumptions of Theorem \ref{Sparse Domination}. In what follows, by an operator $L$ being of type $\omega$ on $L^2(\mathbb R^n)$ we mean that its spectrum is contained in the sector 
\[
S = \{ \lambda \in \mathbb{C} : |\arg(\lambda)| \leq \omega \}
\]
and its resolvent satisfies a bound of the type 
\[
\| (L - \xi I)^{-1} \| \leq \frac{C_\mu}{|\xi|}
\]
for all $\xi$ with $\arg \xi\ge \mu$ with $\mu>\omega$. See for example \cite{Mc}.\\

In what follows, for a subset $E\subset \mathbb R^n$ with finite measure and a measurable function $f$ we denote
$$
\fint_E f(x)\,{\rm d}x=\f{1}{|E|}\int_E f(x)\,{\rm d}x.
$$We now  consider an operator $L$ satisfying the following conditions.

\begin{enumerate}
	\item[(A1)] $L$ is a closed densely-defined operator of type $\omega$ in $L^2(\mathbb R^n)$ with
	$0\leq \omega<\pi/2$, hence, $L$ generates a holomorphic semigroup
	$e^{zL}$, $|\arg(z)|<\pi/2-\omega$. We also assume that $L$ has a bounded $H_\infty$--functional calculus on $L^2(\mathbb R^n)$.
	
	\item[(A2)] There exist $1\le p_0 < q_0\le \infty$, $\epsilon>0$ and $\kappa>0$ such that for $p_0\le p\le q\le q_0$, $t>0$ and $k\in \mathbb N:=\{0,1,2,\ldots\}$ we have
	\begin{equation}\label{Lpq estimate}
		\|(tL)^ke^{-tL}\|_{p\to q} \lesi_{k,p,q} t^{-\f{n}{\kappa}(\f{1}{p}-\f{1}{q})} 
	\end{equation}
	and
	\begin{equation}\label{eq2-Tt}
		\begin{aligned}
			\Big(\fint_{Q}|(tL)^ke^{-tL}f|^q\Big)^{1/q}&\lesi_{k,p,q} \max\Big\{\Big(\f{3^j\ell(Q)}{t^{1/\kappa}}\Big)^n,\Big(\f{3^j\ell(Q)}{t^{1/\kappa}}\Big)^{n/p} \Big\} \\
   	&\hskip 1cm \times\Big(1+\f{t^{1/\kappa}}{\ell(Q)}\Big)^{n/q}\Big(1+\f{3^j\ell(Q)}{t^{1/\kappa}}\Big)^{-n-\epsilon}
		 \Big(\fint_{S_j(Q)}|f|^p\Big)^{1/p}
		\end{aligned}
	\end{equation}
	for all cubes $Q$ and $f\in L^p(S_j(Q))$ with $j\ge 2$, where $S_j(Q)=3^{j}Q\backslash 3^{j-1}Q$ as $j\ge 1$ and $S_0(Q)=Q$.
\end{enumerate}

\begin{rem}
	Recall that a family $\{S_t\}_{t>0}$ is said to satisfy the $L^p$--$L^q$ off-diagonal estimate (of order $\kappa>1$) for some $p,q \in [1, \infty]$ with $p \leq q$ if there exist constants $C, c> 0$ such that for all closed sets $E, F \subset \mathbb{R}^n$, $t > 0$, and $f \in L^2(\mathbb{R}^n) \cap L^p(\mathbb{R}^n)$ supported in $E$, the following estimate holds:
	\begin{equation}\label{eq-defn Lpq estimates}
		\|S_tf\|_{L^q(F)} \leq C t^{  \f{n}{\kappa}\left(\frac{1}{q} - \frac{1}{p}\right) } e^{-\left(\frac{d(E,F)}{ct^{1/\kappa}}\right)^{\f{\kappa}{\kappa-1}}} \|f\|_{L^p(E)}.
	\end{equation}
	For further details and properties of  the $L^p$--$L^q$ off-diagonal estimates we refer to \cite{AM2}. It is easy to see that if $e^{-tL}$ satisfies (A1) and the  $L^{p}$--$L^{q}$ off-diagonal estimate \eqref{eq-defn Lpq estimates}, then $L$ satisfies (A2) with the same $\kappa>0$ and any $\epsilon>0$. 
\end{rem}

For each $0<\alpha<n$, we consider the fractional power $L^{-\alpha/\kappa}$ defined by
\[
L^{-\alpha/\kappa} = \f{1}{\Gamma(\alpha/\kappa)}\int_0^\infty s^{\alpha/\kappa} e^{-sL}\frac{\ud s}{s}.
\]
In the next subsection, we will consider a number of operators satisfying the assumptions (A1) and (A2). We now would like to give two typical examples. Firstly, when $L = -\Delta$ is the Laplacian on $\mathbb R^n$ and $\kappa =2$, then  $L^{-\alpha/2}$ turns out out to be the fractional integral defined by \eqref{eq-fractional integral}. In addition, when \( L \) is a second-order elliptic operator in divergence form with complex coefficients, \( L \) satisfies \eqref{eq-defn Lpq estimates} for some \( 1 < p_0 < 2 < q_0 < \infty \), \( \kappa = 2 \) and any $\epsilon>0$. The weighted estimates for the fractional operator \( L^{-\alpha/\kappa} \) were investigated in \cite{AM1}.  We now take a step further to establish a quantitative two-weight estimate for \( L^{-\alpha/\kappa} \) under the weaker assumptions (A1) and (A2). More precisely, we are able to prove the following results, with the definitions of weight provided in Section 3.

\begin{thm}{\label{weighted estimate for L}}
	Let $L$ satisfy (A1) and (A2) for some $1\le p_0<q_0\le \infty$, $\epsilon>0$ and $\kappa>0$. Let $p_0 <p <q< q_0$, $\alpha = n(\f{1}{p}-\f{1}{q})$. Suppose $\mu,\lambda$ are two weights, denote $u=\mu^{\f{p_0 p}{p_0-p}}$, $v=\lambda^{\frac{q_0 q}{q_0 - q}}$ and suppose further that $u,v \in A_{\infty}$ and $(u,v) \in A_{\frac{1}{p_0}-\f{1}{p}, \f{1}{q}-\f{1}{q_0}}^{\f{\alpha}{n}-\f{1}{p_0}+\f{1}{q_0}}$. Then 
    $$\|L^{-\alpha/\kappa} f\|_{L^{q}(\lambda^{q})} \lesssim [u, v]_{A_{\frac{1}{p_0}-\f{1}{p}, \f{1}{q}-\f{1}{q_0}}^{\f{\alpha}{n}-\f{1}{p_0}+\f{1}{q_0}}}([u]_{A_{\infty}}^{\f{1}{q}}+[v]_{A_{\infty}}^{\f{1}{p^{\prime}}})\|f\|_{L^{p}(\mu^{p})}.$$
\end{thm}

Note that in \cite{BFP} the sharp weighted estimate for the class of non-singular integral operators beyond the Calder\'on-Zygmund theory was investigated. In this perspective, Theorem \ref{weighted estimate for L} presents  the two-weight estimate for the class of \textit{fractional operators beyond the Calder\'on-Zygmund theory}.
\begin{rem}
As we mentioned above, if $L=-\Delta$ the Laplacian on $\mathbb R^n$, the fractional power $L^{-\alpha/2}$ becomes the fractional integral operator $I_{\alpha}$ defined by (\ref{eq-fractional integral}). In this case, Theorem \ref{weighted estimate for L} with  $\lambda=\mu=\omega \in A_{p,q}$ provides a better estimate than the sharp weighted estimate (\ref{Lacey's estimate}) since $[\omega]_{A_\infty} \lesssim [\omega]_{A_{p,q}}.$
 \end{rem}

Since 
$$[\omega^{r}]_{A_{r(q-1)+1}} \leqslant [\omega]_{A_q}^{r}[\omega]_{\rm{RH}_r}^{r},$$
we can obtain another weighted estimate under the assumption of $\omega \in A_{1+\frac{1}{p_0}-\frac{1}{p}} \cap \rm{RH}_{q(\frac{q_0}{q})^{\prime}}$, which provides a quantitative weighted estimate for $L^{-\alpha / \kappa} $ and improves the results in \cite[Theorem 1.3]{AM1}. Specifically, we obtain

\begin{cor}{\label{weighted estimate corollary}}
	Let $L$ satisfy (A1) and (A2) for some $1\le p_0<q_0\le \infty$, $\epsilon>0$ and $\kappa>0$. Let $p_0 <p <q< q_0$, $\alpha = n(\f{1}{p}-\f{1}{q})$. Suppose $\omega \in A_{1+\frac{1}{p_0}-\frac{1}{p}} \cap {\rm RH}_{q(\frac{q_0}{q})^{\prime}}$. Then we have
    $$\|L^{-\alpha/\kappa} f\|_{L^{q}(\omega^{q})} \lesssim ([\omega]_{A_{1+\frac{1}{p_0}-\frac{1}{p}}}[\omega]_{\text{\rm RH}_{q(\frac{q_0}{q})^{\prime}}})^{1+\max\{{\frac{p_0 p}{p - p_0} \cdot \frac{1}{q}, \frac{q_0 q}{q_0-q}  \cdot \frac{1}{p^{\prime}}}\}}\|f\|_{L^{p}(\omega^{p})}.$$
\end{cor}

In this paper, we also investigate the Bloom weighted boundedness for iterative commutators of fractional operators. For the sake of simplicity, we only provide a special case here. One could find a more general Bloom weighted estimate and related discussions in Section 4.

\begin{thm}{\label{Bloom weighted estimate for L}}
   Let $L$ satisfy (A1) and (A2) for some $1\le p_0<q_0\le \infty$, $\epsilon>0$ and $\kappa>0$. Let $p_0 <p <q< q_0$, $\alpha = n(\f{1}{p}-\f{1}{q}), m \in \mathbb{N} \cup \{0\}$ and $b\in L_{\rm loc}^{1}(\mathbb{R}^n)$. Suppose $\omega \in A_{1+\frac{1}{p_0}-\frac{1}{p}} \cap {\rm RH}_{q(\frac{q_0}{q})^{\prime}}$. Then we have
    $$\|(L^{-\alpha/\kappa})_b^m f\|_{L^{q}(\omega^{q})} \lesssim C(\omega)\|b\|_{\text{BMO}} \|f\|_{L^{p}(\omega^p)},
    $$
where 
$$C(\omega) = ([\omega]_{A_{1+\frac{1}{p_0}-\frac{1}{p}}}[\omega]_{\text{\rm RH}_{q(\frac{q_0}{q})^{\prime}}})^{1+\max\{{\frac{p_0 p}{p - p_0} \cdot \frac{1}{q}, \frac{q_0 q}{q_0-q}  \cdot \frac{1}{p^{\prime}}}\}+\max\{p_0 p, \frac{p_0^2 p}{ p-p_0 }, q_0^{\prime} q^{\prime}, \frac{(q_0^{\prime})^2 q^{\prime}}{q^{\prime}-q_0^{\prime}}\}  \cdot m}$$
\end{thm}

To prove Theorem \ref{weighted estimate for L} and Theorem \ref{Bloom weighted estimate for L}, we will establish a sparse domination criterion for a general class of fractional operators, including the classical fractional integrals \eqref{eq-fractional integral} and the class of fractional integrals detailed in \cite{IRV}. See Theorem \ref{Sparse Domination}. Later, we will show that the fractional powers \( L^{-\alpha/\kappa} \) fall within the scope of our theory. See Section \ref{App sec}.

For $1 \leqslant s \leqslant \infty$, we define 
\begin{equation}{\label{sharp grand maximal function}}
\mathcal{M}_{T,s}^{\#}f(x):=\sup_{x \in Q}\operatorname{osc}_{s}(T(f\chi_{\mathbb{R}^n \backslash {3Q}});Q), \quad x\in \mathbb{R}^n,
\end{equation}
as the sharp grand maximal truncation operator, where
$$\operatorname{osc}_{s}(f;Q):=\Big(\frac{1}{|Q|^2}\int_{Q \times Q}|f(x^{\prime})-f(x^{\prime\prime})|^{s}\ud x^{\prime}\ud x^{\prime\prime}\Big)^{1/s}$$ 
and the supremum is taken over all cubes $Q$ containing $x$. 

To present the sparse domination criterion, we also need the following boundedness condition.
\begin{defn}
	Let $1 \leqslant p_0 < \infty$, $0 \leqslant \alpha < \frac{n}{p_0}$. We say that an operator $T$ is \emph{locally weak $L^{p_0} \rightarrow L^{\frac{p_0n}{n-\alpha p_0}}$ bounded} if there exists a non-increasing function $\varphi_{T,p_0,\alpha}:(0,1) \rightarrow [0,\infty)$ such that 
	$$|\{x \in Q:|T(f\chi_Q)(x)| > \varphi_{T,p_0,\alpha}(\lambda)\ave{|f|}_{p_0,Q}|Q|^{\frac{\alpha}{n}}\}| \leqslant \lambda |Q|$$
	holds for all cube $Q$ and $f \in L^{p_0}(Q)$. 
\end{defn}
When $\alpha=0$, this definition is exactly the $W_{p_0}$ property in \cite{LO}. Meanwhile, it should be noted that the usual weak $L^{p_0} \rightarrow L^{\frac{p_0n}{n-\alpha p_0}}$ boundedness implies the locally weak $L^{p_0} \rightarrow L^{\frac{p_0n}{n-\alpha p_0}}$ boundedness.

Below, we present a result on the sparse domination for iterative commutators. We point out that in the case of $m=0$, the following result provide a sparse domination for the underlying operator. For precise definitions of dyadic lattices and sparse families, we refer the reader to Section 2.

\begin{thm}{\label{Sparse Domination}}
	Let $1 \leqslant p_0 < q_0 \leqslant \infty$, $m \in \mathbb{N} \cup \{0\}$. Suppose $T$ is a sublinear operator and both $T$ and $\mathcal{M}_{T,q_0}^{\#}$ are locally weak $L^{p_0} \rightarrow L^{\frac{p_0n}{n-\alpha p_0}}$ bounded. Then there exist $C_{m,n} > 1$ and $0<\lambda_{m,n} < 1$ such that for any $f,g \in L_{c}^{\infty}(\mathbb{R}^n), b \in L_{loc}^{1}(\mathbb{R}^n)$, there exists a dyadic lattice $\mathscr{D}$ and a $\frac{1}{2\cdot 3^{n}}$-sparse collection of cubes $\mathcal{S} \subset \mathscr{D}$ such that 
	\begin{align}
		\int_{\mathbb{R}^n}|T_b^{m}f||g| \ud x \leqslant C&\Big(\sum\limits_{Q \in \mathcal{S}}\ave{|b-\ave{b}_Q|^{m}|f|}_{p_0,Q}\ave{|g|}_{q_0^{\prime},Q}{|Q|}^{1+\frac{\alpha}{n}}\nonumber \\
		&+ \sum\limits_{Q \in \mathcal{S}}\ave{|f|}_{p_0,Q}\ave{|b-\ave{b}_Q|^{m}|g|}_{q_0^{\prime},Q}{|Q|}^{1+\frac{\alpha}{n}}\Big), \nonumber
	\end{align}
	where
	\begin{equation}{\label{Coefficient in Sparse Domination}} 
		C:=C_{m,n}(\varphi_{T,p_0,\alpha}(\lambda_{m,n})+\varphi_{\mathcal{M}_{T,q_0}^{\#},p_0,\alpha}(\lambda_{m,n})).
	\end{equation} 
\end{thm}

With this criteria at hand, we can prove the two-weight estimate for such operators as detailed in Section 3. As a byproduct, we provide a general framework for the quantitative weighted bound for sparse forms which improves the results of \cite{Li-2017, FH, LLO}. See Theorem \ref{weighted estimate} below. Moreover, we also prove the Bloom weighted estimate for commutators of such operators with a similar strategy as in \cite{LLO}. See Section 4.

\subsection{Some applications}{\label{some application}}
We would like to provide some typical examples of important operators satisfying (A1) and (A2).

\begin{enumerate}[{\rm (a)}]
	\item If $L$ is a non-negative self adjoint operator on $L^2(\mathbb R^n)$ satisfying the Poisson upper bound of order $\kappa$, i.e., there exist $C>0$ and $\kappa, \epsilon_0>0$ such that 
	\[
	|e^{-tL}(x,y)|\le \f{C}{t^{n/\kappa}}\Big(\f{t^{1/\kappa}}{|x-y|+t^{1/\kappa}}\Big)^{n+\epsilon_0}
	\] 
	for all $t>0$ and $x,y\in \mathbb R^n$, where $e^{-tL}(x,y)$ denote the kernel of $e^{-tL}$, then $L$ satisfies (A1) and (A2) with $p_0=1$ and $q_0=\infty$ and any $\epsilon \in (0,\epsilon_0)$. See for example \cite{BD}. In particular, when $L=-\Delta$ the Laplacian, the conditions (A1) and (A2) are fulfilled with $\kappa=2$ and $p_0=1, q_0=\infty$. In this case, it is known that the fractional power $(-\Delta)^{-\alpha/2}$ coincides with the fractional integral operator $I_\alpha$ defined by \eqref{eq-fractional integral}. See for example \cite{St}.
	
	\item \textit{The second-order elliptic divergence operator:} let $L=-\text{div}(A(x)\nabla)$ be the second-order elliptic divergence operator, where $A(x)$ is an $n \times n$ matrix defined on $\mathbb{R}^n$ with complex $L^\infty$-coefficients satisfying the ellipticity (or "accretivity") conditions
	\[
	\lambda |\xi|^2 \leq \Re(A\xi \cdot \overline{\xi}) \quad \text{and} \quad |A\xi \cdot \overline{\zeta}| \leq \Lambda |\xi| \cdot |\zeta|,
	\]
	for $\xi, \zeta \in \mathbb{C}^n$ and for some $\lambda, \Lambda$ such that $0 < \lambda < \Lambda < \infty$. 
	Then $L$ satisfies (A1) and \eqref{eq-defn Lpq estimates} for $\kappa=2$ and some $1\le p_0<q_0\le \infty$. See \cite{Au} for the precise values of $p_0$ and $q_0$. Consequently, $L$ satisfies (A1) and (A2). It emphasizes that the weighted inequalities for the fractional power $L^{-\alpha/2}$ were investigated in \cite{AM1}, but not the sharp weighted estimates. Similarly, if $L$  is a higher order divergence form elliptic operator  with complex bounded measurable coefficients as investigated in \cite{Ddy}, then $L$ also satisfies (A1) and (A2). See \cite[Theorem 3.2]{Ddy} for precise values of $\kappa$ and $p_0, q_0$.
	
	\item \textit{Generalized Hardy operator:} consider the following generalized Hardy operators on $\mathbb{R}^n, n\in \mathbb{N}$,
	\begin{equation}\label{defn-La}
		L = (-\Delta)^{\beta/2}+a|x|^{-\beta} \quad \text{with} \quad 0<\beta<\min\{2, n\} \ \ \text{and} \ \ \quad a\geq a^*,
	\end{equation}
	where
	$$
	a^*=-\f{2^\beta\Gamma((n+\beta)/4)^2}{\Gamma((n-\beta)/4)^2}.
	$$
	It is known that $L$ is a nonnegative self adjoint operator and hence it satisfies (A1). In addition, it was proved in \cite{BD} that $L$ satisfies (A2) with some $1\le p_0< q_0\le \infty$, $\kappa=\beta$ and any $\epsilon\in(0,\beta)$. See \cite{BD} for the precise values of $p_0$ and $q_0$.

	\item \textit{Hardy operators in a half-space:} Consider the following generalized Schr\"odinger operators in $L^{2}(\mathbb{R}^n_+)$,
	\begin{equation}\label{defn-La}
		L= (-\Delta)^{\beta/2}_{\mathbb R^n_+} + \lambda x_n^{-\beta},
	\end{equation}
	where $\mathbb R_+^n:=\mathbb R^{n-1}\times\mathbb R_+$, $\mathbb R_+:=(0,\infty)$, $\beta\in(0,2]$,  $(-\Delta)_{\mathbb R_+^n}$ is the Laplacian with a Dirichlet boundary condition on $(\mathbb R_+^n)^c$,
	and, when $\beta\in(0,2)$, $(-\Delta)^{\beta/2}_{\Rd_+}$ denotes the regional fractional Laplacian, which was introduced, e.g., in \cite{BBC}.
	It is known that $L$ is a non-negative self adjoint operator and hence it satisfies (A1). In addition, it was proved in \cite{BD} that $L$ satisfies (A2) with some $1\le p_0< q_0\le \infty$, $\kappa=\beta$ and any $\epsilon \in (0,\beta)$. See \cite{BM} for the precise values of $p_0$ and $q_0$.
	
\end{enumerate}

The organization of the paper is as follows. Section 2 will prove Theorem \ref{Sparse Domination}. Two-weight estimates and the Bloom weighted estimates for the operators satisfying the assumptions of Theorem \ref{Sparse Domination} will be given in Section 3 and Section 4 respectively. Finally, the proof of Theorem {\ref{weighted estimate for L}} will be detailed in Section 5.

\section{A Sparse Domination for Commutators}\label{Sec 2}
 This section is dedicated to prove Theorem \ref{Sparse Domination}, which extends the conclusion in \cite{LLO} to the off-diagonal case. Firstly, we provide some definitions that will be used later on. 
 
 We now recall the definition of dyadic lattice and its basic properties as outlined in \cite{LN}. By a cube in $\mathbb{R}^n$, we mean a half-open cube $Q=\mathop{\Pi}\limits_{i=1}^{n}[x_i-h,x_i+h)$, where $h > 0$. Denote by $\ell(Q)$ the sidelength of $Q$. For $\lambda > 0$, define $\lambda Q = \mathop{\Pi}\limits_{i=1}^{n}[x_i - \lambda h,x_i + \lambda h)$. Given a cube $Q_0 \subset \mathbb{R}^n$, let $\mathscr{D}(Q_{0})$ denote the set of all dyadic cubes with respect to $Q_0$, i.e., the cubes obtained by repeated subdivision of $Q_0$ and each of its descendants into $2^{n}$ congruent subcubes.

A dyadic lattice $\mathscr{D}$ in $\mathbb{R}^n$ is any collection of cubes such that,

(i) if $Q \in \mathscr{D}$, then $\mathscr{D}(Q) \subset \mathscr{D}$;

(ii) every two cubes $Q,Q^{\prime} \in \mathscr{D}$ have a common ancestor, i.e., there exists $R \in \mathscr{D}$ such that $Q,Q^{\prime} \in \mathscr{D}(R)$;

(iii) for every compact set $K \subset \mathbb{R}^n$, there exists a cube $Q \in \mathscr{D}$ containing $K$.

The following theorem is known as the "three lattice theorem".
\begin{thm}\label{thm2.1}\emph{(}\cite[Theorem 3.1]{LN}\emph{)}
For every dyadic lattice $\mathscr{D}$, there exist $3^{n}$ dyadic lattices $\mathscr{D}^1,\cdots,\mathscr{D}^{3^n}$ such that
$$\{3Q:Q \in \mathscr{D}\}=\mathop{\cup}_{j=1}^{3^n}\mathscr{D}^j$$
and for every cube $Q \in \mathscr{D}$ and $j=1,\cdots,3^n$, there exists a unique cube $R \in \mathscr{D}^j$ of sidelength $\ell(R)=3\ell(Q)$ containing $Q$.  
\end{thm} 

We say that a family $\mathcal{S}$ of cubes in $\mathbb{R}^n$ is $\eta-$sparse, $0< \eta < 1$, if for every $Q \in \mathcal{S}$, there exists a  measurable pairwise disjoint family $\{E_Q\}_{Q \in \mathcal{S}}$ such that $E_Q \subset Q$ and $|E_Q| \geqslant \eta|Q|$.  
For the sake of simplicity, in this paper, we call $\mathcal{S}$ sparse if there exists an $\eta$ such that $\mathcal{S}$ is $\eta$-sparse.

Throughout this paper, for a cube $Q \subset \mathbb{R}^n$, $r > 0$ and $f \in L_{loc}^{r}(\mathbb{R}^n), u \in L_{loc}^{1}(\mathbb{R}^n)$, we denote 
$$\ave{f}_{r,Q}:=\Big(\frac{1}{|Q|}\int_{Q}f^{r}\Big)^{\frac{1}{r}}, \quad \text{and} \quad \ave{f}_{r,Q}^{u}:=\Big(\frac{1}{u(Q)}\int_{Q}f^{r} u \Big)^{\frac{1}{r}}.$$ 
In particular, when $r=1$, we simplify the notation as $\ave{f}_{Q}:=\ave{f}_{1,Q}, \ave{f}_{Q}^{u}:=\ave{f}_{1,Q}^{u}$. We also define the maximal function by
$$\operatorname{M}_{r}f:=\sup_{Q \in \mathscr{D}}\ave{|f|}_{r,Q}\chi_Q, \quad \text{and} \quad \operatorname{M}_{r,u}f:=\sup_{Q \in \mathscr{D}}\ave{|f|}_{r,Q}^{u}\chi_Q .$$

The proof of this theorem adopts a similar strategy as in \cite{LLO}, we will show the main steps for the completeness of this paper. In fact, Theorem \ref{Sparse Domination} is a direct corollary of the following two results.

\begin{thm}{\label{Sparse Domination-2}}
Under the assumptions of Theorem \ref{Sparse Domination} we have
$$ \int_{\mathbb{R}^n}|T_b^{m}f||g| \ud x \leqslant C\sum\limits_{k=0}^{m}\Big(\sum\limits_{Q \in \mathcal{S}}\ave{|b-\ave{b}_Q|^{m-k}|f|}_{p_0,Q}\ave{|b-\ave{b}_Q|^{k}|g|}_{q_0^{\prime},Q}{|Q|}^{1+\frac{\alpha}{n}}\Big),
$$
where C is given by (\ref{Coefficient in Sparse Domination}).
\end{thm}

\begin{lem}\emph{(}\cite[Lemma 3.4]{LLO}\emph{)}
Let $1 \leqslant r,t < \infty$, $m \in \mathbb{N}$. Suppose $f, g \in L_{c}^{\infty}(\mathbb{R}^n), b \in L_{loc}^1(\mathbb{R}^n)$. Fix a cube $Q \in \mathscr{D}$ and for $0 \leqslant k \leqslant m$ define
$$c_k:=\ave{|b-\ave{b}_Q|^{m-k}|f|}_{r,Q}\ave{|b-\ave{b}_Q|^{k}|g|}_{t,Q}.$$
Then we have $c_k \leqslant c_0 + c_m.$
\end{lem}

\begin{proof}[Proof of Theorem \ref{Sparse Domination-2}]
Suppose the supports of $f$ and $g$ are contained in $Q \in \mathscr{D}$. For $k=0,\cdots,m$, denote
$$\eta_k:=(b-\ave{b}_{3Q})^k f\chi_{3Q}.$$
Now we consider the sets  
$$\Omega_k:=\{x \in Q:|T(\eta_k)(x)|>\varphi_{T,p_0,\alpha}(\frac{1}{(m+1)6^{n+2}})\ave{|\eta_k|}_{p_0,3Q}|3Q|^{\frac{\alpha}{n}}\},$$
and
$$\mathbb{M}_k:=\{x \in Q:|\mathcal{M}_{T,q_0}^{\#}(\eta_k)(x)|>\varphi_{\mathcal{M}_{T,q_0}^{\#},p_0,\alpha}(\frac{1}{(m+1)6^{n+2}})\ave{|\eta_k|}_{p_0,3Q}|3Q|^{\frac{\alpha}{n}}\}.$$
By the locally weak $L^{p_0} \rightarrow L^{\frac{p_0n}{n-\alpha p_0}}$ boundedness of $T$ and $\mathcal{M}_{T,q_0}^{\#}$, we have
$$|\Omega_k| \leqslant \frac{1}{3(m+1)2^{n+2}}|Q|,\quad |\mathbb{M}_k| \leqslant \frac{1}{3(m+1)2^{n+2}}|Q|.$$ 
Since the maximal function $M_{p_0}$ is weak $L^{p_0}$-bounded with constant independent of $p_0$, there exists a $c_{n,m}>0$ such that 
$$M_k:=\{x \in Q: M_{p_0}(\eta_k)(x) > c_{n,m}\ave{\eta_k}_{p_0,3Q}\}$$
also satisfies
$$|M_k| \leqslant \frac{1}{3(m+1)2^{n+2}}|Q|.$$
Therefore, setting
$$\Omega:=\mathop{\cup}\limits_{k=0}^m(\Omega_k \cup \mathbb{M}_k \cup M_k),$$
we have $|\Omega| \leqslant \frac{1}{2^{n+2}}|Q|.$ 
Applying the local Calder\'on-Zygmund decomposition to $\chi_{\Omega}$ at height $\frac{1}{2^{n+1}}$, we obtain a family of pairwise disjoint cubes $\mathcal{S}_Q \subset \mathscr{D}(Q)$ such that $|\Omega \backslash \mathop{\cup}\limits_{Q^{\prime} \in \mathcal{S}_Q}Q^{\prime}|=0$ and for each $Q^{\prime} \in \mathcal{S}_{Q}$,
\begin{equation}{\label{Sparseness}}
\frac{1}{2^{n+1}}|Q^{\prime}| < |Q^{\prime} \cap \Omega| \leqslant \frac{1}{2}|Q^{\prime}|.
\end{equation}
It follows that
\begin{equation}{\label{Sparseness-2}}
\sum\limits_{Q^{\prime} \in \mathcal{S}_Q}|Q^{\prime}| \leqslant 2^{n+1}|\Omega| \leqslant \frac{1}{2}|Q|.
\end{equation}
Meanwhile, we have
\begin{align}
\int_{\mathbb{R}^n}|T_{b}^{m}f||g|\ud x = & \int_{Q}|T_{b}^{m}(f\chi_{3Q})||g|\ud x \nonumber \\
 \leqslant & \int_{Q \backslash \mathop{\bigcup}\limits_{Q^{\prime} \in \mathcal{S}_Q}Q^{\prime}}|T_{b}^{m}(f\chi_{3Q})||g| \ud x \nonumber \\
& + \sum\limits_{Q^{\prime} \in \mathcal{S}_Q}\int_{Q^{\prime}}|T_{b}^{m}(f\chi_{3Q\backslash 3Q^{\prime}})||g| \ud x \nonumber \\
& + \sum\limits_{Q^{\prime} \in \mathcal{S}_Q}\int_{Q^{\prime}}|T_{b}^{m}(f\chi_{3Q^{\prime}})||g| \ud x \nonumber \\
 := & H_1 + H_2 + \sum\limits_{Q^{\prime} \in \mathcal{S}_Q}\int_{Q^{\prime}}|T_{b}^{m}(f\chi_{3Q^{\prime}})||g| \ud x .\nonumber
\end{align}
Our goal is to prove that
\begin{align}{\label{claim}}
H_1 + H_2 \leqslant C\sum\limits_{k=0}^{m}\ave{|b-\ave{b}_{3Q}|^{m-k}|f|}_{p_0,3Q}\ave{|b-\ave{b}_{3Q}|^{k}|g|}_{{q_0}^{\prime},3Q}|Q|^{1+\frac{\alpha}{n}},\end{align}
where C is given by (\ref{Coefficient in Sparse Domination}). For $H_1$, note that for any $c \in \mathbb{C}$, $T_{b}^{m}(f)=T_{b-c}^{m}(f)$. Together with the definition of $\Omega_k$, we have 
\begin{align}
H_1 = & \int_{Q \backslash \mathop{\bigcup}\limits_{Q^{\prime} \in \mathcal{S}_Q}Q^{\prime}}|T_{b}^{m}(f\chi_{3Q})||g| \ud x \nonumber \\
\leqslant & \sum\limits_{k=0}^{m}{m \choose k} \int_{Q \backslash \mathop{\bigcup}\limits_{Q^{\prime} \in \mathcal{S}_Q}Q^{\prime}}|T((b-\ave{b}_{3Q})^{m-k}f\chi_{3Q})||b-\ave{b}_{3Q}|^{k}|g| \ud x \nonumber \\
\leqslant & C_1 \sum\limits_{k=0}^{m}\ave{|b-\ave{b}_{3Q}|^{m-k}|f|}_{p_0,3Q}\ave{|b-\ave{b}_{3Q}|^{k}|g|}_{1,3Q}|Q|^{1+\frac{\alpha}{n}}, \nonumber
\end{align}
where $C_1:=2^{m}3^{2n}\varphi_{T,p_0,\alpha}\Big(\frac{1}{(m+1)6^{n+2}}\Big)$.

For $H_2$, fix $Q^{\prime} \in \mathcal{S}_Q$, for $k=0,1,\cdots,m$, denote
$$\xi_k:=(b-\ave{b}_{3Q})^{m-k}f\chi_{3Q^{\prime}},\quad \psi_k:=T((b-\ave{b}_{3Q})^{m-k}f\chi_{3Q \backslash 3Q^{\prime}}),$$
and we consider the sets
$$\tilde{\Omega}_k:=\{x \in Q^{\prime}:|T(\xi_k)(x)| > \varphi_{T,p_0,\alpha}\Big(\frac{1}{4(m+1)3^{n}}\Big)\ave{|\xi_k|}_{p_0,3Q^{\prime}}|3Q^{\prime}|^{\frac{\alpha}{n}}\}. 
$$
Set $\tilde{\Omega}:=\mathop{\cup}_{k=0}^{m}\tilde{\Omega}_k$, we have $|\tilde{\Omega}| \leqslant \frac{1}{4}|Q^{\prime}|.$ Now, define the good part of the cube $Q^{\prime}$ as 
$$G_{Q^{\prime}}:=Q^{\prime} \backslash (\Omega \cup \tilde{\Omega}).$$
By (\ref{Sparseness}), we have $|G_{Q^{\prime}}| \geqslant \frac{1}{4}|Q^{\prime}|$. Meanwhile, for each $y \in G_{Q^{\prime}}$, by the definition of $\Omega_k$, $M_k$, and $\tilde{\Omega}_k$, we have 
\begin{align}
|\psi_k(y)| \leqslant & |T((b-\ave{b}_{3Q})^{m-k}f\chi_{3Q})| + |T((b-\ave{b}_{3Q})^{m-k}f\chi_{3Q^{\prime}})| \nonumber \\
\leqslant & \varphi_{T,p_0,\alpha}\Big(\frac{1}{(m+1)6^{n+2}}\Big)\ave{|(b-\ave{b}_{3Q})^{m-k}f|}_{p_0,3Q}|3Q|^{\frac{\alpha}{n}} \nonumber \\
& + \varphi_{T,p_0,\alpha}\Big(\frac{1}{4(m+1)3^{n}}\Big)\ave{|(b-\ave{b}_{3Q})^{m-k}f|}_{p_0,3Q^{\prime}}|3Q^{\prime}|^{\frac{\alpha}{n}} \nonumber \\
\leqslant & 2\cdot 3^{n} c_{n,m} \cdot \varphi_{T,p_0,\alpha}\Big(\frac{1}{(m+1)6^{n+2}}\Big)\ave{|(b-\ave{b}_{3Q})^{m-k}f|}_{p_0,3Q}|Q|^{\frac{\alpha}{n}}. \label{upper bound}
\end{align}
Then
\begin{align}
 \int_{Q^{\prime}}|T_{b}^{m}&(f\chi_{3Q \backslash 3Q^{\prime}})||g| \ud x \\
 = &\frac{1}{|G_{Q^{\prime}}|}\int_{G_{Q^{\prime}}}\int_{Q^{\prime}}|T_{b}^{m}(f\chi_{3Q\backslash 3Q^{\prime}})(x)||g(x)| \ud x \ud y \nonumber \\
\leqslant & 2^{m}\frac{1}{|G_{Q^{\prime}}|}\int_{G_{Q^{\prime}}}\int_{Q^{\prime}}\sum\limits_{k=0}^{m}|\psi_k(x)||(b-\ave{b}_{3Q})^{m-k}(x)||g(x)| \ud x \ud y \nonumber \\
\leqslant & 2^{m}\frac{1}{|G_{Q^{\prime}}|}\int_{G_{Q^{\prime}}}\int_{Q^{\prime}}\sum\limits_{k=0}^{m}|\psi_k(x)-\psi_k(y)||(b-\ave{b}_{3Q})^{m-k}(x)||g(x)| \ud x \ud y \nonumber \\
& +   2^{m}\frac{1}{|G_{Q^{\prime}}|}\int_{G_{Q^{\prime}}}\int_{Q^{\prime}}\sum\limits_{k=0}^{m}|\psi_k(y)||(b-\ave{b}_{3Q})^{m-k}(x)||g(x)| \ud x \ud y \nonumber \\
:= &  I_1 + I_2. \nonumber 
\end{align}
For $I_1$, using the H\"{o}lder inequality and the definition of $\mathbb{M}_k$,
\begin{align}
I_1 & \leqslant  4 \cdot 2^{m}\sum\limits_{k=0}^{m}(\frac{1}{|Q^{\prime}|^2}\int_{Q^{\prime}}\int_{Q^{\prime}}|\psi_k(x)-\psi_k(y)|^{q_0} \ud x\ud y)^{\frac{1}{q_0}}\ave{|(b-\ave{b}_{3Q})|^{k}|g|}_{{q_0}^{\prime},Q^{\prime}}|Q^{\prime}| \nonumber \\
& \leqslant  4 \cdot 2^{m}\sum\limits_{k=0}^{m}\operatorname{osc}_{q_0}(T(\eta_{m-k}\chi_{3Q \backslash 3Q^{\prime}});Q^{\prime})\ave{|(b-\ave{b}_{3Q})|^{k}|g|}_{{q_0}^{\prime},Q^{\prime}}|Q^{\prime}| \nonumber \\
& \leqslant \widetilde{C_1}\sum\limits_{k=0}^{m}\ave{|(b-\ave{b}_{3Q})^{m-k}f|}_{p_0,3Q}\ave{|(b-\ave{b}_{3Q})^{m-k}g|}_{{q_0}^{\prime},Q^{\prime}}|Q|^{\frac{\alpha}{n}}|Q^{\prime}|, \nonumber
\end{align}
where $\widetilde{C_1} =  2^{m+2}3^{n} \varphi_{\mathcal{M}_{T,q_0}^{\#},p_0,\alpha}\Big(\frac{1}{(m+1)6^{n+2}}\Big)$. 

We now estimate $I_2$. By (\ref{upper bound}), we have
$$I_2 \leqslant \widetilde{C_2}\sum\limits_{k=0}^{m}\ave{|(b-\ave{b}_{3Q})^{m-k}f|}_{p_0,3Q}\ave{|(b-\ave{b}_{3Q})^{m-k}g|}_{1,Q^{\prime}}|Q|^{\frac{\alpha}{n}}|Q^{\prime}|,
$$
where $\widetilde{C_2} =  2^{m+1}3^{n}c_{n,m} \cdot \varphi_{T,p_0,\alpha}\Big(\frac{1}{(m+1)6^{n+2}}\Big)$. 

Using the H\"{o}lder inequality, for any $q \in [1,\infty)$, we have
$$\sum\limits_{Q^{\prime} \in \mathcal{S}_Q}\ave{|h|}_{q,Q^{\prime}}|Q^{\prime}| \leqslant \ave{|h|}_{q,Q}|Q|.$$
Therefore,
\begin{align}
H_2 & = \sum\limits_{Q^{\prime} \in \mathcal{S}_Q}\int_{Q^{\prime}}|T_{b}^{m}(f\chi_{3Q\backslash 3Q^{\prime}})||g| \ud x \nonumber \\
& \leqslant  C_2 \ave{|(b-\ave{b}_{3Q})^{m-k}f|}_{p_0,3Q}|Q|^{\frac{\alpha}{n}}\sum\limits_{Q^{\prime} \in \mathcal{S}_Q}\ave{|(b-\ave{b}_{3Q})^{m-k}g|}_{{q_0}^{\prime},Q^{\prime}}|Q^{\prime}| \nonumber \\
& \leqslant C_2 \ave{|(b-\ave{b}_{3Q})^{m-k}f|}_{p_0,3Q}\ave{|(b-\ave{b}_{3Q})^{m-k}g|}_{{q_0}^{\prime},3Q}|Q|^{1 + \frac{\alpha}{n}}, \nonumber
\end{align}
where $C_2$ is given by
$$C_2: = C_{n,m}\Big(\varphi_{\mathcal{M}_{T,q_0}^{\#},p_0,\alpha}\Big(\frac{1}{(m+1)6^{n+2}}\Big)+\varphi_{T,p_0,\alpha}\Big(\frac{1}{(m+1)6^{n+2}}\Big)\Big).
$$
So far, by ({\ref{Sparseness-2}}) and ({\ref{claim}}), we have found a $\frac{1}{2}$-sparse family $\mathcal{S}_Q$  satisfying 
\begin{align}
\int_{\mathbb{R}^n}|T_{b}^{m}f||g|\ud x \leqslant & C\sum\limits_{k=0}^{m}\ave{|b-\ave{b}_{3Q}|^{m-k}|f|}_{p_0,3Q}\ave{|b-\ave{b}_{3Q}|^{k}|g|}_{{q_0}^{\prime},3Q}|Q|^{1+\frac{\alpha}{n}} \nonumber \\
& + \sum\limits_{Q^{\prime} \in \mathcal{S}_Q}\int_{Q^{\prime}}|T_{b}^{m}(f\chi_{3Q^{\prime}})||g| \ud x, \nonumber
\end{align}
where $C$ is given by (\ref{Coefficient in Sparse Domination}). By iteration, we get a $\frac{1}{2}$-sparse family $\mathcal{F}$, such that
$$
\int_{\mathbb{R}^n}|T_{b}^{m}f||g|\ud x \leqslant C\sum\limits_{k=0}^{m}\Big(\sum\limits_{Q \in \mathcal{F}}\ave{|b-\ave{b}_{3Q}|^{m-k}|f|}_{p_0,3Q}\ave{|b-\ave{b}_{3Q}|^{k}|g|}_{{q_0}^{\prime},3Q}|Q|^{1+\frac{\alpha}{n}}\Big). 
$$
Now, by applying the "three lattice theorem" to $\mathscr{F}$ and taking the largest one among the $3^n$ dyadic lattices as $\mathcal{S}$, we obtain the desired result. 
\end{proof}

\section{Two-weight estimates for general fractional sparse forms}\label{Sec: shar estimate}

In this section, we provide a two-weight estimate for general fractional sparse forms. As a corollary, we prove a two-weight estimate for operators satisfying the assumptions of Theorem \ref{Sparse Domination}. 

In this paper, we refer to a non-negative locally integrable function $\omega$ as a weight. For $1<p<\infty$, the set $A_p$ is composed of weights that satisfy 
$$[w]_{A_{p}}:=\sup_{Q: \,\text{cube in}\, \mathbb{R}^n} \ave{w}_Q\ave{w^{1-p^{\prime}}}_Q^{p-1}<\infty.
$$
For $p=\infty$, the set $A_\infty$ is composed of weights that satisfy 
$$
[w]_{A_{\infty}}:=\sup_{Q: \,\text{cube in}\, \mathbb{R}^n} \frac{1}{w(Q)} \int_{Q} M\left(w \chi_Q\right),
$$where $M$ denotes the Hardy-Littlewood maximal function. For $1 \leqslant r \leqslant \infty$, the set $\rm{RH}_r$ is composed of weights that satisfy 
$$
[w]_{\rm{RH}_r}:=\sup_{Q: \,\text{cube in}\, \mathbb{R}^n} \frac{\ave{\omega}_{r,Q}}{\ave{\omega}_{Q}}. $$

Now, we define the two-weight $A_{\beta,\gamma}^{\alpha}$ used in the weighted estimates in this paper. Let $\alpha, \beta, \gamma \in \mathbb{R}$, the set $A_{\beta,\gamma}^{\alpha}$ is composed of two-weights $(\omega,\sigma)$ that satisfy 
\begin{equation}{\label{two-weight}}
	[\omega,\sigma]_{A_{\beta,\gamma}^{\alpha}}:= \sup_{Q: \,\text{cube in}\, \mathbb{R}^n}|Q|^{\alpha}\omega(Q)^{\beta} \sigma(Q)^{\gamma} < \infty.
\end{equation}
For any sparse family $\mathcal{S}$, we say that two-weight $(\omega,\sigma)$ belongs to $A_{\beta,\gamma}^{\alpha}(\mathcal{S})$ if 
\begin{equation}
	\sup_{Q \in \mathcal{S}}|Q|^{\alpha}\omega(Q)^{\beta} \sigma(Q)^{\gamma} < \infty \nonumber.
\end{equation}

\begin{rem}
	Only when $\alpha \leqslant 0$ and $ \alpha+\beta+\gamma \geqslant 0$, the two-weight defined above is non-trivial. Indeed, when we consider the cube increasing in size such that $\ell(Q)$ tends to infinity, we obtain $\alpha \leqslant 0$. We then consider the cube shrinking to a fixed point, since the above two-weight constants $A_{\beta,\gamma}^{\alpha}$ can be rewritten as 
	$$[\omega,\sigma]_{A_{\beta,\gamma}^{\alpha}}= \sup_{Q: \,\text{cube in}\, \mathbb{R}^n}|Q|^{\alpha + \beta + \gamma} \ave{\omega}_Q^{\beta} \ave{\sigma}_Q^{\gamma},$$
	by Lebesgue differentiation theorem, we obtain $\alpha+\beta+\gamma \geqslant 0$.
\end{rem}

It should be note that when $\sigma:=\omega^{1-p^{\prime}}$, $\alpha = -p$, $\beta=1$, $\gamma=p-1$, the two-weights defined in (\ref{two-weight}) return to the usual $A_p$ weights.

Throughout this paper, we write $A \lesssim B$ if there exists $C > 0$ (which possibly depends on $n,r,p,q,s,p_0,q_0$ or $T$) such that $A \leqslant C B$, write $A \gtrsim B$ if $A \geqslant C B$, and write $A \simeq B$ if $A \lesssim B$ and $A \gtrsim B$.

Our main result in this section is formulated by the following theorem.
\begin{thm}{\label{weighted estimate}}
Let $1  \leqslant p \leqslant q < s \leqslant \infty$, $r \in (0,p)$ and $\frac{1}{p}-\frac{1}{q} \leqslant \alpha < \frac{1}{r}-\frac{1}{s}$. Suppose $\omega, \sigma$ are two weights and $\mathcal{S} \subset \mathscr{D}$ is a sparse family. Suppose $\mathcal{N}$ is the best constant such that for any $f \in L^{p}(\omega)$, $g \in L^{q^{\prime}}(\sigma)$ the following inequality holds
$$\sum\limits_{Q \in \mathcal{S}}\ave{|f|}_{r,Q}\ave{|g|}_{s^{\prime},Q}{|Q|}^{1+\alpha} \leqslant \mathcal{N}\|f\|_{L^{p}(\omega)}\|g\|_{L^{q^{\prime}}(\sigma)}.
$$
If $\mathcal{N} < \infty$, we have $(\omega^{\frac{r}{r - p}},\sigma^{\frac{s^{\prime}}{s^{\prime} - q^{\prime}}}) \in {A}_{\frac{1}{r}-\frac{1}{p},\frac{1}{q}-\frac{1}{s}}^{\alpha - \frac{1}{r} + \frac{1}{s}}$, and in this case
\begin{equation}{\label{weighted estimate formula}}
1 \leqslant \mathcal{N} \lesssim [u,v]_{{A}_{\frac{1}{r}-\frac{1}{p},\frac{1}{q}-\frac{1}{s}}^{\alpha - \frac{1}{r} + \frac{1}{s}}}
\Bigg\{\begin{aligned} 
&[u]_{A_\infty}^{{1-\frac{1}{(p^\prime)^2}}}[v]_{A_\infty}^{\frac{1}{(p^\prime)^2}} + [u]_{A_\infty}^{\frac{1}{p^2}}[v]_{A_\infty}^{1-\frac{1}{p^2}} & \quad &\text{ if }  p=q, \alpha > 0 \\
&[u]_{A_\infty}^{\frac{1}{q}} + [v]_{A_\infty}^{\frac{1}{p^{\prime}}} & \quad &\text{ otherwise }
\end{aligned} ,
\end{equation}
where $u=\omega^{\frac{r}{r - p}}$, $v=\sigma^{\frac{s^{\prime}}{s^{\prime} - q^{\prime}}}$.
\end{thm}
\begin{rem}
This theorem provides a general framework for the quantitative weights bound for sparse forms. On one hand, it generalizes \cite[Theorem 4.1]{LLO} to the two-weight scenario; on the other hand, it extends \cite[Theorem 1.2]{Li-2017} to the off-diagonal case. Moreover, this estimate can be used to deduce the quantitative weighted bound for some sparse operators. In particular, it covers the main results in \cite{FH}, and we will show the process at the end of this section.
\end{rem}
Combining the above theorem with the case of $m=0$ in Theorem \ref{Sparse Domination}, we can obtain the following corollary.
\begin{cor}{\label{two-weight estimate}}
Let $1 \leqslant p_0 < p \leqslant q  < q_0 \leqslant \infty$, $n(\frac{1}{p}-\frac{1}{q}) \leqslant \alpha < n(\frac{1}{p_0}-\frac{1}{q_0})$. Suppose $T$ is a sublinear operator and both $T$ and $\mathcal{M}_{T,q_0}^{\#}$ are locally weak $L^{p_0} \rightarrow L^{\frac{p_0n}{n-\alpha p_0}}$ bounded. Suppose $(\omega^{\frac{p_0}{p_0 - p}},\sigma^{\frac{q_0^{\prime}}{q_0^{\prime} - q^{\prime}}}) \in {A}_{\frac{1}{p_0}-\frac{1}{p},\frac{1}{q}-\frac{1}{q_0}}^{\f{\alpha}{n} - \frac{1}{p_0} + \frac{1}{q_0}}$ and $\mathcal{N}$ is the best constant such that for every $f,g \in L_{c}^{\infty}(\mathbb{R}^n)$ the following inequality holds
$$ \int_{\mathbb{R}^n}|Tf||g| \ud x \leqslant \mathcal{N}\|f\|_{L^{p}(\omega)}\|g\|_{L^{q^{\prime}}(\sigma)}.$$
Then we have the estimate 
\begin{equation}
\mathcal{N} \lesssim [u,v]_{{A}_{\frac{1}{p_0}-\frac{1}{p},\frac{1}{q}-\frac{1}{q_0}}^{\frac{\alpha}{n} - \frac{1}{p_0} + \frac{1}{q_0}}}
\Bigg\{\begin{aligned} 
&[u]_{A_\infty}^{{1-\frac{1}{(p^\prime)^2}}}[v]_{A_\infty}^{\frac{1}{(p^\prime)^2}} + [u]_{A_\infty}^{\frac{1}{p^2}}[v]_{A_\infty}^{1-\frac{1}{p^2}} & \quad &\text{ if }  p=q, \alpha > 0 \\
&[u]_{A_\infty}^{\frac{1}{q}} + [v]_{A_\infty}^{\frac{1}{p^{\prime}}} & \quad &\text{ otherwise }
\end{aligned} , \nonumber
\end{equation}
where $u=\omega^{\frac{p_0}{p_0 - p}}$, $v=\sigma^{\frac{q_0^{\prime}}{q_0^{\prime} - q^{\prime}}}$.
\end{cor}

To prove Theorem \ref{weighted estimate}, we need the following lemmas. The first lemma comes from \cite{LLO}, which is a slight generalization of the result in \cite{Li-2017}.

\begin{lem}\emph{(}\cite[Lemma 4.4]{LLO}\emph{)}{\label{Lemma1}}
Let $1 < p \leqslant q < s \leqslant \infty$, $r \in (0,p)$ and $\lambda_Q \geqslant 0$ for any $Q \in \mathscr{D}$. Suppose $\omega, \sigma$ are two weights and $\mathcal{S} \subset \mathscr{D}$ is a sparse family.  Suppose $\mathcal{N}$ is the best constant such that for any $f \in L^{p}(\omega)$, $g \in L^{q^{\prime}}(\sigma)$ the following inequality holds
$$\sum\limits_{Q \in \mathcal{S}}\ave{|f|}_{r,Q}\ave{|g|}_{s^{\prime},Q}\lambda_Q \leqslant \mathcal{N}\|f\|_{L^{p}(\omega)}\|g\|_{L^{q^{\prime}}(\sigma)}.
$$
Denote $u:=\omega^{\frac{r}{r-p}}$, $v:=\sigma^{\frac{s^{\prime}}{s^{\prime}-q^{\prime}}}$, set
$$\tau_Q:=\ave{u}_{Q}^{\frac{1}{r}-1}\ave{v}_Q^{-\frac{1}{s}}\frac{\lambda_Q}{|Q|},\quad Q \in \mathscr{D}$$
and for $R \in \mathscr{D}$ define 
$$T_Rf:=\sum\limits_{\substack{Q \in \mathcal{S} \\ Q \subset R}} \tau_Q \ave{f}_Q\chi_Q. $$
Then
\begin{equation}{\label{best constant estimate}}
\mathcal{N} \simeq \zeta + \zeta^{*} := \sup_{R \in \mathcal{S}} \frac{\|T_R(u)\|_{L^{q}(v)}}{{u(R)}^{1/p}} + \sup_{R \in \mathcal{S}} \frac{\|T_R(v)\|_{L^{p^{\prime}}(u)}}{{v(R)}^{1/q^{\prime}}}.
\end{equation} 
\end{lem}

In the case of $p=1$, we need the following lemma.

\begin{lem}{\label{Lemma2}}
Let $1 \leqslant q < s \leqslant \infty$, $r \in (0,1)$ and $\lambda_Q \geqslant 0$ for any $Q \in \mathscr{D}$. Suppose $\omega, \sigma$ are two weights and $\mathcal{S} \subset \mathscr{D}$ is a sparse family.  Suppose $\mathcal{N}$ is the best constant such that for any $f \in L^{1}(\omega)$, $g \in L^{q^{\prime}}(\sigma)$ the following inequality holds
$$\sum\limits_{Q \in \mathcal{S}}\ave{|f|}_{r,Q}\ave{|g|}_{s^{\prime},Q}\lambda_Q \leqslant \mathcal{N}\|f\|_{L^{1}(\omega)}\|g\|_{L^{q^{\prime}}(\sigma)}.
$$
Then
\begin{equation}{\label{best constant estimate2}}
\mathcal{N} \lesssim \zeta = \sup_{R \in \mathcal{S}} \frac{\|T_R(u)\|_{L^{q}(v)}}{{u(R)}},
\end{equation} 
where $u=\omega^{\frac{r}{r-1}}$, $v,T_R$ are defined in Lemma {\ref{Lemma1}}.
\end{lem}
\begin{proof}
According to the definitions of $u$ and $v$, $\mathcal{N}$ is the best constant satisfying the following estimate:
$$\sum\limits_{Q \in \mathcal{S}}\ave{|f|}_{r,Q}\ave{|g|}_{s^{\prime},Q}\lambda_Q \leqslant \mathcal{N}\||f|^{r}u^{-1}\|_{L^{1/r}(u)}^{1/r}\||g|^{s^{\prime}}v^{-1}\|_{L^{q^{\prime}/s^{\prime}}(v)}^{1/s^{\prime}},
$$
which is equivalent to
$$\sum\limits_{Q \in \mathcal{S}}\ave{|f|}_{r,Q}^{u}\ave{|g|}_{s^{\prime},Q}^{v} \lambda_Q \ave{u}_{Q}^{\frac{1}{r}}\ave{v}_{Q}^{\frac{1}{s^{\prime}}} \leqslant \mathcal{N}\|f\|_{L^{1}(u)} \|g\|_{L^{q^{\prime}}(v)}.
$$

Suppose $\mathcal{N}^{\prime}$ is the best constant such that for any $f \in L^{1}(u)$, $g \in L^{q^{\prime}}(v)$ the following inequality holds
$$\sum\limits_{Q \in \mathcal{S}}\ave{|f|}_{r,Q}^{u} \ave{|g|}_{Q}^{v} \lambda_Q \ave{u}_{Q}^{\frac{1}{r}}\ave{v}_{Q}^{\frac{1}{s^{\prime}}} \leqslant \mathcal{N}^{\prime}\|f\|_{L^{1}(u)} \|g\|_{L^{q^{\prime}}(v)}.
$$
It is not difficult to verify that $\mathcal{N}$ and $\mathcal{N}^{\prime}$ are comparable. Indeed, the inequality $\mathcal{N}^{\prime} \leqslant \mathcal{N}$ comes directly from the H\"{o}lder inequality. On the other hand, since
\begin{align}
\sum\limits_{Q \in \mathcal{S}}\ave{|f|}_{r,Q}^{u}\ave{|g|}_{s^{\prime},Q}^{v}  \lambda_Q \ave{u}_{Q}^{\frac{1}{r}}\ave{v}_{Q}^{\frac{1}{s^{\prime}}} 
\leqslant & \sum\limits_{Q \in \mathcal{S}}\ave{|f|}_{r,Q}^{u} \ave{\operatorname{M}_{s^{\prime},v}(g)}_{Q}^{v} \lambda_Q \ave{u}_{Q}^{\frac{1}{r}}\ave{v}_{Q}^{\frac{1}{s^{\prime}}} \nonumber \\
\leqslant &\, \mathcal{N}^{\prime}\|f\|_{L^{1}(u)} \|\operatorname{M}_{s^{\prime},v}(g)\|_{L^{q^{\prime}}(v)} \nonumber \\
\lesssim &\, \mathcal{N}^{\prime}\|f\|_{L^{1}(u)} \|g\|_{L^{q^{\prime}}(v)}, \nonumber
\end{align}
we have $\mathcal{N} \lesssim \mathcal{N}^{\prime}$. As a result, we only need to prove estimate (\ref{best constant estimate2}) for constant $\mathcal{N}^{\prime}$. 

To this end, we will use the stopping time argument, which was introduced by Li and Sun in \cite{LS}, and further improved by Dimi\'an, Hormozi, and Li in \cite{DHL}. Without loss of generality, we may assume that the sparse family $\mathcal{S}$ has a maximal cube $Q_0$. We construct the stopping time family $\mathscr{F}$ inductively. Let $\mathscr{F}_0 := \{Q_0\}$ and
$$\mathscr{F}_{k} := \bigcup_{F \in \mathscr{F}_{k-1}}\{F^{\prime} \subset F : F^{\prime} \text{\, is the maximal cube in\,} \mathcal{S} \text{\,satisfying\,} \ave{|f|}_{r,F^{\prime}}^{u} > 2 \ave{|f|}_{r,F}^{u}\}.$$ 
Then the stopping time family is defined by $\mathscr{F}:=\bigcup_{k=0}^ {\infty}\mathscr{F}_k$. It is easy to deduce from the above construction that  
\begin{equation}
\sum\limits_{F \in \mathscr{F}}\ave{|f|}_{r,F}^{u}u(F) \lesssim \|\operatorname{M}_{r,u}(f)\|_{L^{1}\left(u\right)} \lesssim \left\|f\right\|_{L^{1}\left(u\right)}. \nonumber
\end{equation}
We use $\pi_{\mathscr{F}}(Q)$ to represent the stopping parents of $Q$, that is, the minimal cube containing $Q$ in $\mathscr{F}$. According to the definition, we have $\ave{|f|}_{r,Q}^{u} \leqslant 2\ave{|f|}_{r,\pi_{\mathscr{F}}(Q)}^{u}$.

Now, let's return to the estimate of $\mathcal{N}^{\prime}$. Since
\begin{align}
\sum\limits_{Q \in \mathcal{S}}\ave{|f|}_{r,Q}^{u}\ave{|g|}_{Q}^{v} \lambda_Q \ave{u}_{Q}^{\frac{1}{r}}\ave{v}_{Q}^{\frac{1}{s^{\prime}}}
\lesssim & \sum\limits_{F \in \mathscr{F}}\ave{|f|}_{r,F}^{u}\sum_{\substack{Q \in \mathcal{S} \\ \pi_{\mathscr{F}}(Q)= F }}\ave{|g|}_{Q}^{v} \lambda_Q \ave{u}_{Q}^{\frac{1}{r}}\ave{v}_{Q}^{\frac{1}{s^{\prime}}} \nonumber \\
\leqslant & \sum\limits_{F \in \mathscr{F}}\ave{|f|}_{r,F}^{u}\int_{\mathbb{R}^n}T_{F}(u)|g|v \ud x \nonumber \\
\leqslant & \sum\limits_{F \in \mathscr{F}}\ave{|f|}_{r,F}^{u}\|T_{F}(u)\|_{L^{q}(v)}\|g\|_{L^{q^{\prime}}(v)} \nonumber \\
\leqslant & \, \zeta \sum\limits_{F \in \mathscr{F}}\ave{|f|}_{r,F}^{u}u(F)\|g\|_{L^{q^{\prime}}(v)} \nonumber \\
\leqslant & \, \zeta \|f\|_{L^{1}(u)} \|g\|_{L^{q^{\prime}}(v)},\nonumber
\end{align}
we obtain estimate (\ref{best constant estimate2}) for constant $\mathcal{N}^{\prime}$.
\end{proof}

In the following, we will estimate the right hand terms of (\ref{best constant estimate}) and (\ref{best constant estimate2}), and we need another two lemmas.
\begin{lem}\emph{(}\cite[Proposition 2.2]{COV}\emph{)}{\label{Lemma3}}
Let $p \in [1,\infty)$ and $\lambda_Q \geqslant 0$ for any $Q \in \mathscr{D}$. Suppose $\omega$ is a weight. Then 
$$\|\sum\limits_{Q \in \mathscr{D}}\lambda_Q\chi_Q\|_{L^{p}(\omega)} \simeq \Big(\sum\limits_{Q \in \mathscr{D}}\lambda_Q\Big(\frac{1}{\omega(Q)}\sum\limits_{\substack{Q^{\prime} \in \mathscr{D} \\ Q^{\prime} \subset Q}}\lambda_{Q^{\prime}}\omega(Q^{\prime})\Big)^{p-1}\omega(Q)\Big)^{1/p}.
$$ 
\end{lem}

\begin{lem}\emph{(}\cite[Lemma 4.2]{FH}\emph{)}{\label{Lemma4}}
Let $\alpha, \beta, \gamma \geqslant 0$ with $\alpha + \beta +\gamma \geqslant 1$. Suppose $\omega, \sigma$ are two weights and $\mathcal{S} \subset \mathscr{D}$ is a sparse family. For a cube $R$, the following are hold true.

(i) In the case of $\alpha > 0$, one has the universal estimate
$$ \sum\limits_{\substack{Q \in \mathcal{S} \\ Q \subset R}} |Q|^{\alpha}\sigma(Q)^{\beta}\omega(Q)^{\gamma} \lesssim |R|^{\alpha}\sigma(R)^{\beta}\omega(R)^{\gamma}.
$$

(ii) In the case of $\alpha = 0$, one still has the weaker inequality
$$ \sum\limits_{\substack{Q \in \mathcal{S} \\ Q \subset R}} \sigma(Q)^{\beta}\omega(Q)^{\gamma} \lesssim [\sigma]_{A_\infty}^{\beta}[\omega]_{A_\infty}^{\gamma}\sigma(R)^{\beta}\omega(R)^{\gamma}.
$$
\end{lem}

\begin{proof}[Proof of Theorem \ref{weighted estimate}]
For the sake of simplicity in writing, in this proof we use $[u,v]$ to represent the two-weight constant $[u,v]_{{A}_{\frac{1}{r}-\frac{1}{p},\frac{1}{q}-\frac{1}{s}}^{\alpha - \frac{1}{r} + \frac{1}{s}}}$.

We first prove the necessity, for each $Q \in \mathcal{S}$,
\begin{align}
|Q|^{\alpha-\frac{1}{r}+\frac{1}{s}}u(Q)^{\frac{1}{r}-\frac{1}{p}}v(Q)^{\frac{1}{q}-\frac{1}{s}} & = u(Q)^{-\frac{1}{p}}v(Q)^{\frac{1}{q}} \ave{\omega^{\frac{1}{r-p}}\chi_Q}_{r,Q}\ave{\sigma^{\frac{1}{s^{\prime}-q^{\prime}}}\chi_Q}_{s^{\prime},Q}|Q|^{1+\alpha} \nonumber \\
& \leqslant u(Q)^{-\frac{1}{p}}v(Q)^{\frac{1}{q}} \sum\limits_{Q \in \mathcal{S}}\ave{\omega^{\frac{1}{r-p}}\chi_Q}_{r,Q}\ave{\sigma^{\frac{1}{s^{\prime}-q^{\prime}}}\chi_Q}_{s^{\prime},Q}|Q|^{1+\alpha} \nonumber \\
& \leqslant \mathcal{N} u(Q)^{-\frac{1}{p}}v(Q)^{\frac{1}{q}}\|\omega^{\frac{1}{r-p}}\chi_Q\|_{L^{p}(\omega)}\|\sigma^{\frac{1}{s^{\prime}-q^{\prime}}}\chi_Q\|_{L^{q^{\prime}}(\sigma)} \nonumber \\
& = \mathcal{N}. \nonumber
\end{align}
The above estimate implies $(u, v) \in {A}_{\frac{1}{r}-\frac{1}{p},\frac{1}{q}-\frac{1}{s}}^{\alpha - \frac{1}{r} + \frac{1}{s}}$ and the lower bound in ({\ref{weighted estimate formula}}). To prove the upper bound, we use Lemma {\ref{Lemma1}} and Lemma {\ref{Lemma2}} to obtain
$$\mathcal{N} \simeq \zeta + \zeta^{*}, $$
in the case of $1<p<\infty$, and 
$$\mathcal{N} \lesssim \zeta, $$
in the case of $p=1$, where
$$\zeta = \sup_{R \in \mathcal{S}} \frac{1}{{u(R)}^{1/p}}\Big\|\sum\limits_{\substack{Q \in \mathcal{S} \\ Q \subset R}} \ave{u}_{Q}^{\frac{1}{r}}\ave{v}_Q^{-\frac{1}{s}}|Q|^{\alpha}\chi_Q\Big\|_{L^{q}(v)},$$
$$\zeta^{*} = \sup_{R \in \mathcal{S}} \frac{1}{{v(R)}^{1/q^{\prime}}}\Big\|\sum\limits_{\substack{Q \in \mathcal{S} \\ Q \subset R}}\ave{u}_{Q}^{\frac{1}{r}-1}\ave{v}_Q^{\frac{1}{s^{\prime}}}|Q|^{\alpha}\chi_Q\Big\|_{L^{p^{\prime}}(u)}.$$
For $\zeta$, using Lemma {\ref{Lemma3}}, we have 
$$\zeta \simeq \sup_{R \in \mathcal{S}} \frac{1}{{u(R)}^{1/p}}\Big(\sum\limits_{\substack{Q \in \mathcal{S} \\ Q \subset R}}\ave{u}_{Q}^{\frac{1}{r}}\ave{v}_Q^{-\frac{1}{s}}|Q|^{\alpha}\Big(\frac{1}{v(Q)}\sum\limits_{\substack{Q^{\prime} \in \mathcal{S} \\ Q^{\prime} \subset Q}}\ave{u}_{Q^{\prime}}^{\frac{1}{r}}\ave{v}_{Q^{\prime}}^{-\frac{1}{s}}|{Q^{\prime}}|^{\alpha}v(Q^{\prime})\Big)^{q-1}v(Q)\Big)^{1/q}.$$
Now, we estimate the inner sum
\begin{align}
&\sum\limits_{\substack{Q^{\prime} \in \mathcal{S} \\ Q^{\prime} \subset Q}}\ave{u}_{Q^{\prime}}^{\frac{1}{r}}\ave{v}_{Q^{\prime}}^{-\frac{1}{s}}|{Q^{\prime}}|^{\alpha}v(Q^{\prime}) \nonumber \\
 = & \sum\limits_{\substack{Q^{\prime} \in \mathcal{S} \\ Q^{\prime} \subset Q}}u(Q^{\prime})^{\frac{1}{r}}v(Q^{\prime})^{\frac{1}{s^{\prime}}}|{Q^{\prime}}|^{\alpha-\frac{1}{r}+\frac{1}{s}}\nonumber \\
\lesssim & [u,v]^{\delta}\sum\limits_{\substack{Q^{\prime} \in \mathcal{S} \\ Q^{\prime} \subset Q}} u(Q^{\prime})^{\frac{1}{r}-\delta(\frac{1}{r}-\frac{1}{p})}v(Q^{\prime})^{\frac{1}{s^{\prime}}-\delta(\frac{1}{q}-\frac{1}{s})}|Q^{\prime}|^{(\alpha-\frac{1}{r}+\frac{1}{s})(1-\delta)}, \nonumber
\end{align} 
where $\delta$ satisfies the following four conditions:
\begin{align}
\frac{1}{r}-\delta(\frac{1}{r}-\frac{1}{p}) \geqslant 0 &\Leftrightarrow \delta \leqslant \frac{p}{p - r}, \nonumber \\
\frac{1}{s^{\prime}}-\delta(\frac{1}{q}-\frac{1}{s}) \geqslant 0 &\Leftrightarrow \delta \leqslant \frac{(s - 1) q}{s - q}, \nonumber \\
(\alpha-\frac{1}{r}+\frac{1}{s})(1-\delta) \geqslant 0 &\Leftrightarrow \delta \geqslant 1, \nonumber \\
1 + \alpha - \delta(\alpha-\frac{1}{p}+\frac{1}{q}) \geqslant 1 &\Leftrightarrow \alpha \geqslant \delta(\alpha-\frac{1}{p}+\frac{1}{q}). \nonumber
\end{align}
Note that when $p=q, \alpha > 0$ or $p=q=1$ and only $\delta=1$ satisfies the above four conditions,  we can only apply (ii) in Lemma \ref{Lemma4}. Otherwise, we can choose $\delta>1$, so that (i) in Lemma \ref{Lemma4} can be applied. In summary, we have
\begin{equation}
\sum\limits_{\substack{Q^{\prime} \in \mathcal{S} \\ Q^{\prime} \subset Q}}\ave{u}_{Q^{\prime}}^{\frac{1}{r}}\ave{v}_{Q^{\prime}}^{-\frac{1}{s}}|{Q^{\prime}}|^{\alpha}v(Q^{\prime}) \lesssim \zeta_{inn}u(Q)^{\frac{1}{r}-\delta(\frac{1}{r}-\frac{1}{p})}v(Q)^{\frac{1}{s^{\prime}}-\delta(\frac{1}{q}-\frac{1}{s})}|Q|^{(\alpha-\frac{1}{r}+\frac{1}{s})(1-\delta)}, \label{inner sum}
\end{equation}
where
$$\zeta_{inn}=[u,v]^{\delta} \cdot \Big\{\begin{aligned} 
&[u]_{A_\infty}^{\frac{1}{p}}[v]_{A_\infty}^{\frac{1}{q^{\prime}}} & \quad &\text{ if }  p=q, \alpha > 0 \, \text{or} \, p=q=1\\
&1 & \quad &\text{ otherwise }
\end{aligned}.$$
Here, we additionally require $\delta$ to satisfy $q \geqslant \delta (q-1)$. Then we have 
\begin{align}
&\sum\limits_{\substack{Q \in \mathcal{S} \\ Q \subset R}}\ave{u}_{Q}^{\frac{1}{r}}\ave{v}_Q^{-\frac{1}{s}}|Q|^{\alpha}\Big(\frac{1}{v(Q)}\sum\limits_{\substack{Q^{\prime} \in \mathcal{S} \\ Q^{\prime} \subset Q}}\ave{u}_{Q^{\prime}}^{\frac{1}{r}}\ave{v}_{Q^{\prime}}^{-\frac{1}{s}}|{Q^{\prime}}|^{\alpha}v(Q^{\prime})\Big)^{q-1}v(Q)  \nonumber \\
\lesssim & \zeta_{inn}^{q-1}\sum\limits_{\substack{Q \in \mathcal{S} \\ Q \subset R}}u(Q)^{\frac{q}{p}+ (\frac{1}{r}-\frac{1}{p})(q-\delta(q-1))}v(Q)^{
(\frac{1}{q}-\frac{1}{s})(q-\delta(q-1))}|Q|^{(\alpha-\frac{1}{r}+\frac{1}{s})(q-\delta(q-1))} \nonumber \\
\leqslant & \zeta_{inn}^{q-1} [u,v]^{q-\delta(q-1)} \sum\limits_{\substack{Q \in \mathcal{S} \\ Q \subset R}}u(Q)^{\frac{q}{p}} \nonumber \\
\leqslant & \zeta_{inn}^{q-1} [u,v]^{q-\delta(q-1)} u(R)^{\frac{q}{p}-1}\sum\limits_{\substack{Q \in \mathcal{S} \\ Q \subset R}}u(Q) \nonumber \\
\leqslant & \zeta_{inn}^{q-1} [u,v]^{q-\delta(q-1)} [u]_{A_\infty}u(R)^{\frac{q}{p}},\nonumber 
\end{align}
where in last inequality we used (ii) in Lemma \ref{Lemma4}. Integrating the above estimates together, we have 
\begin{align}
\zeta &\lesssim [u,v]\cdot[u]_{A_\infty}^{\frac{1}{q}}\Bigg\{\begin{aligned} 
&([u]_{A_\infty}^{\frac{1}{p}}[v]_{A_\infty}^{\frac{1}{q^{\prime}}})^{\frac{q-1}{q}} & \quad &\text{ if }  p=q, \alpha > 0 \\
&1 & \quad &\text{ otherwise }
\end{aligned} \nonumber\\
&= [u,v]\cdot\Bigg\{\begin{aligned} 
&[u]_{A_\infty}^{{1-\frac{1}{(p^\prime)^2}}}[v]_{A_\infty}^{\frac{1}{(p^\prime)^2}} & \quad &\text{ if }  p=q, \alpha > 0 \\
&[u]_{A_\infty}^{\frac{1}{q}} & \quad &\text{ otherwise }
\end{aligned}. \label{zeta}
\end{align}

In the case of $1<p<\infty$, we use the same method for $\zeta^{*}$. Note that the inner sum of $\zeta$ and $\zeta^{*}$ are the same. So we can directly apply (\ref{inner sum}) to obtain
\begin{align}
& \Big\|\sum\limits_{\substack{Q \in \mathcal{S} \\ Q \subset R}}\ave{u}_{Q}^{\frac{1}{r}-1}\ave{v}_Q^{\frac{1}{s^{\prime}}}|Q|^{\alpha}\chi_Q \Big \|_{L^{p^{\prime}}(u)} \nonumber \\
\simeq & \Big(\sum\limits_{\substack{Q \in \mathcal{S} \\ Q \subset R}}\ave{u}_{Q}^{\frac{1}{r}-1}\ave{v}_Q^{\frac{1}{s^{\prime}}}|Q|^{\alpha}\Big(\frac{1}{u(Q)}\sum\limits_{\substack{Q^{\prime} \in \mathcal{S} \\ Q \subset Q^{\prime}}}\ave{u}_{Q^{\prime}}^{\frac{1}{r}-1}\ave{v}_{Q^{\prime}}^{\frac{1}{s^{\prime}}}|{Q^{\prime}}|^{\alpha}u(Q^{\prime})\Big)^{p^{\prime}-1}u(Q)\Big)^{1/{p^{\prime}}} \nonumber \\
\lesssim & [u,v] \cdot [v]_{A_{\infty}}^{1/{p^{\prime}}}v(R)^{\frac{1}{q^{\prime}}} \Bigg\{\begin{aligned} 
&([u]_{A_\infty}^{\frac{1}{p}}[v]_{A_\infty}^{\frac{1}{q^{\prime}}})^{\frac{p^{\prime}-1}{p^{\prime}}} & \quad &\text{ if }  p=q, \alpha > 0 \\
&1 & \quad &\text{ otherwise }
\end{aligned}. \nonumber 
\end{align}
Therefore, 
\begin{align}
\zeta^{*} &\lesssim [u,v][v]_{A_{\infty}}^{1/{p^{\prime}}}\Bigg\{\begin{aligned} 
&([u]_{A_\infty}^{\frac{1}{p}}[v]_{A_\infty}^{\frac{1}{q^{\prime}}})^{\frac{p^{\prime}-1}{p^{\prime}}} & \quad &\text{ if }  p=q, \alpha > 0 \\
&1 & \quad &\text{ otherwise }
\end{aligned} \nonumber\\
&= [u,v]\cdot\Bigg\{\begin{aligned} 
&[u]_{A_\infty}^{\frac{1}{p^2}}[v]_{A_\infty}^{1-\frac{1}{p^2}} & \quad &\text{ if }  p=q, \alpha > 0 \\
&[v]_{A_\infty}^{\frac{1}{p^{\prime}}} & \quad &\text{ otherwise }
\end{aligned}. \label{zeta*}
\end{align}
Combining estimates (\ref{zeta}) and (\ref{zeta*}) together, we obtain the desired result.
\end{proof}

At the end of this section, we will show that the main result in \cite{FH} can be deduced from our results. We define the sparse operators $A_{\mathcal{S}}^{r,\alpha}$ as follows:
$$A_{\mathcal{S}}^{r,\alpha}(f):=\Big(\sum\limits_{Q \in \mathcal{S}}\Big(|Q|^{-\alpha}\int_{Q}f)^{r}\chi_Q\Big)^{\frac{1}{r}}.$$ 
\begin{cor}\emph{(}\cite[Theorem 1.1]{FH}\emph{)}
Let $1 < p \leqslant q < \infty, 0 < r < \infty$, and $0 < \alpha \leqslant 1$. Let $\omega,\sigma \in A_{\infty}$ be two weights. Then $A_{\mathcal{S}}^{r,\alpha}(\cdot \sigma)$ maps $L^{p}(\sigma) \rightarrow L^{q}(\omega)$ if and only if $(\omega,\sigma) \in A_{\frac{1}{q},\frac{1}{p^{\prime}}}^{-\alpha}$, and in this case
\begin{equation}{\label{previous result}}
\begin{aligned}
1 & \leqslant \frac{\|A_{\mathcal{S}}^{r,\alpha}(\cdot\sigma)\|_{L^{p}(\sigma) \rightarrow L^{q}(\omega)}}{[\omega,\sigma]_{A_{\frac{1}{q},\frac{1}{p^{\prime}}}^{-\alpha}}} \nonumber \\
& \lesssim \Bigg\{\begin{aligned} 
&[\omega]_{A_\infty}^{\frac{1}{r}(1-\frac{r}{p})^2}[\sigma]_{A_\infty}^{\frac{1}{r}(1-(1-\frac{r}{p})^2)}+[\omega]_{A_\infty}^{\frac{1}{r}(1-(\frac{r}{p})^2)}[\sigma]_{A_\infty}^{\frac{1}{r}(\frac{r}{p})^2} & \quad &\text{ if }   p=q > r, \alpha < 1 \\
&[\omega]_{A_\infty}^{(\frac{1}{r}-\frac{1}{p})_{+}} + [\sigma]_{A_\infty}^{\frac{1}{q}} & \quad &\text{ otherwise } 
\end{aligned}, \nonumber 
\end{aligned}
\end{equation}
where $x_{+}:=\max (x,0)$ in the exponent.
\end{cor}
\begin{proof}
The proofs of necessity and the lower bound of (\ref{previous result}) are the same as  ones in Theorem \ref{weighted estimate} and hence we omit the details here. For the upper bound, we first consider the case when $r \leqslant p$. In this case, we can apply duality argument to obtain
\begin{align}
\|A_{\mathcal{S}}^{r,\alpha}(f\sigma)\|_{L^{q}(\omega)} & = \Big\|\sum\limits_{Q \in \mathcal{S}}|Q|^{-\alpha r}\Big(\int_{Q}f\sigma \Big)^{r}\chi_Q \Big\|_{L^{q/r}(\omega)}^{1/r} \nonumber\\
& =  \sup_{\|g\|_{L^{(q/r)^{\prime}}(\omega)} \leqslant 1} \Big(\sum\limits_{Q \in \mathcal{S}}\ave{(f\sigma)^{r}}_{\frac{1}{r},Q}\ave{g\omega}_{1,Q}|Q|^{1+ r -\alpha r}\Big)^{\frac{1}{r}} \nonumber\\
& \leqslant \mathcal{N}^{\frac{1}{r}}\sup_{\|g\|_{L^{(q/r)^{\prime}}(\omega)} \leqslant 1}\Big(\|(f\sigma)^{r}\|_{L^{p/r}(\sigma^{1-p})}\|g\omega\|_{L^{(q/r)^{\prime}}(\omega^{1-(q/r)^{\prime}})}\Big)^{\frac{1}{r}} \nonumber\\
& = \mathcal{N}^{\frac{1}{r}}\|f\|_{L^{p}(\sigma)}, \nonumber
\end{align}
where $\mathcal{N}$ is the best constant such that for any $f \in L^{p/r}(\sigma^{1-p}), g\in L^{(q/r)^{\prime}}(\omega^{1-(q/r)^{\prime}})$ the following inequality holds
$$ \sum\limits_{Q \in \mathcal{S}}\ave{f}_{\frac{1}{r},Q}\ave{g}_{1,Q}|Q|^{1+ r -\alpha r} \leqslant \mathcal{N}\|f\|_{L^{p/r}(\sigma^{1-p})}\|g\|_{L^{(q/r)^{\prime}}(\omega^{1-(q/r)^{\prime}})}.$$
By Theorem \ref{weighted estimate}, we have
\begin{equation}
\mathcal{N} \lesssim [\sigma,\omega]_{A_{r-r/p,r/q}^{-\alpha r}}\Bigg\{\begin{aligned} 
&[\sigma]_{A_\infty}^{1-(1-\frac{r}{p})^2}[\omega]_{A_\infty}^{(1-\frac{r}{p})^2}+[\sigma]_{A_\infty}^{(\frac{r}{p})^2}[\omega]_{A_\infty}^{1-(\frac{r}{p})^2} &\quad &\text{ if }  p=q , \alpha < 1 \\
&[\sigma]_{A_\infty}^{\frac{r}{q}} + [\omega]_{A_\infty}^{(1-\frac{r}{p})_{+}}   & \quad & \text{otherwise}
\end{aligned}. \nonumber 
\end{equation}
Then the desired result can be obtained by noting $\|A_{\mathcal{S}}^{r,\alpha}(\cdot\sigma)\|_{L^{p}(\sigma) \rightarrow L^{q}(\omega)} \leqslant  \mathcal{N}^{\frac{1}{r}}$. 

In the case of $r > p$, it suffices to prove 
$$
\|A_{\mathcal{S}}^{r,\alpha}(f\sigma)\|_{L^{q}(\omega)} \lesssim [\omega,\sigma]_{A_{\frac{1}{q},\frac{1}{p^{\prime}}}^{-\alpha}}[\sigma]_{A_{\infty}}^{\frac{1}{q}}\|f\|_{L^{p}(\sigma)}.
$$
Note that $A_{\mathcal{S}}^{r,\alpha}(f) \leqslant A_{\mathcal{S}}^{p,\alpha}(f)$, the above result is deduced from the case of $p=r$.
\end{proof}

\section{Bloom weighted estimates for fractional operators}\label{Sec: bloom estiamte}
In this section, we will consider the Bloom weighted estimates for operators satisfying the assumptions of Theorem {\ref{Sparse Domination}}. Firstly, we give some definitions needed in our estimates. Suppose $\mu,\lambda$ are two weights. The Bloom weight $\nu$ is defined by 
\begin{equation}{\label{Bloom weight}}
\nu:=({\mu}/{\lambda})^{\frac{1}{m}}.
\end{equation}
Moreover, we define the Bloom weighted BMO-seminorm as
$$\|b\|_{BMO_{\nu}}:=\sup_{Q: \,\text{cube in}\, \mathbb{R}^n}\frac{1}{\nu(Q)}\int_{Q}|b-\ave{b}_Q|.$$ 

The main theorem of this section will be presented as follows.

\begin{thm}{\label{bloom estimate}}
	Let $1 \leqslant p_0 < p < q < q_0 \leqslant \infty$ with  $\frac{1}{q} = \frac{1}{p} - \frac{\alpha}{n}$, $m \in \mathbb{N} \cup \{0\}$, $b \in L_{loc}^{1}(\mathbb{R}^n)$. Suppose $T$ is a sublinear operator and both $T$ and $\mathcal{M}_{T,q_0}^{\#}$ are locally weak $L^{p_0} \rightarrow L^{\frac{p_0n}{n-\alpha p_0}}$ bounded. Suppose $\mu, \lambda$ are two weights such that $\mu^{\frac{q_0 q}{q_0-q}}, \lambda^{\frac{q_0 q}{q_0-q}} \in A_{1+\frac{q_0 q}{q_0-q}\frac{p-p_0}{p_0 p}}$. Then we have
    $$
\|T_b^{m}f\|_{L^q(\lambda^q)} \lesssim C(\mu,\lambda)\|b\|_{\emph{BMO}_{\nu}} \|f\|_{L^{p}(\mu^{p})}.
    $$
Here $\nu$ is the Bloom weight defined by (\ref{Bloom weight}), and 
$$C(\mu,\lambda) = C_1(\mu,\lambda) + C_2(\mu,\lambda),$$
where
\begin{align}
C_1(\mu,\lambda) = &[\lambda^{\frac{q_0 q}{q_0-q}}]_{A_{1+\frac{q_0 q}{q_0-q}\frac{p-p_0}{p_0 p}}}^{\frac{q_0-q}{q_0 q}(1+\max\{{\frac{p_0 p}{p - p_0} \cdot \frac{1}{q}, \frac{q_0 q}{q_0-q}  \cdot \frac{1}{p^{\prime}}}\})} \label{C1}\\
&\cdot [\mu^{p}]_{A_{p/p_0}}^{\max\{1,\frac{p_0}{p-p_0}\}(p_0m -\floor{p_0m}+\frac{\floor{p_0m}^2}{2p_0 m}+\frac{\floor{p_0m}}{2p_0 m})}   \nonumber\\
&\cdot [\lambda^{p}]_{A_{p/p_0}}^{\max\{1,\frac{p_0}{p-p_0}\}(\floor{p_0m}-\frac{\floor{p_0m}^2}{2p_0 m}-\frac{\floor{p_0m}}{2p_0 m})}, \nonumber 
\end{align}

\begin{align}
C_2(\mu,\lambda) = &[\mu^{\frac{q_0 q}{q_0-q}}]_{A_{1+\frac{q_0 q}{q_0-q}\frac{p-p_0}{p_0 p}}}^{\frac{q_0-q}{q_0 q}(1+\max\{{\frac{p_0 p}{p - p_0} \cdot \frac{1}{q}, \frac{q_0 q}{q_0-q}  \cdot \frac{1}{p^{\prime}}}\})} \nonumber \\
&\cdot [\mu^{-q^{\prime}}]_{A_{q^{\prime}/{q_0^{\prime}}}}^{\max\{1,\frac{q_0^{\prime}}{q^{\prime} - q_0^{\prime}}\}(q_0^{\prime}m -\floor{q_0^{\prime}m}+\frac{\floor{q_0^{\prime}m}^2}{2q_0^{\prime} m}+\frac{\floor{q_0^{\prime}m}}{2q_0^{\prime} m})}  \nonumber\\
& \cdot [\lambda^{-q^{\prime}}]_{A_{q^{\prime}/{q_0^{\prime}}}}^{\max\{1,\frac{q_0^{\prime}}{q^{\prime} - q_0^{\prime}}\}(\floor{q_0^{\prime}m} - \frac{\floor{q_0^{\prime}m}^2}{2q_0^{\prime} m} - \frac{\floor{q_0^{\prime}m}}{2q_0^{\prime} m})}. \nonumber 
\end{align}
\end{thm}
Before coming to the proof, we would like to make some discussion on the above Theorem. If we restrict it to the case of $\mu=\lambda$, we will obtain the following results.
\begin{cor}{\label{bloom estimate1}}
	Let $1 \leqslant p_0 < p < q < q_0 \leqslant \infty$ with  $\frac{1}{q} = \frac{1}{p} - \frac{\alpha}{n}$, $m \in \mathbb{N} \cup \{0\}$, $b \in L_{loc}^{1}(\mathbb{R}^n)$. Suppose $T$ is a sublinear operator and both $T$ and $\mathcal{M}_{T,q_0}^{\#}$ are locally weak $L^{p_0} \rightarrow L^{\frac{p_0n}{n-\alpha p_0}}$ bounded. Suppose $\omega$ is a weight satisfying $\omega^{\frac{q_0 q}{q_0-q}} \in A_{1+\frac{q_0 q}{q_0-q}\frac{p-p_0}{p_0 p}}$. Then, we have
    $$
\|T_b^{m}f\|_{L^q(\omega^q)} \lesssim C(\omega)\|b\|_{\text{BMO}} \|f\|_{L^{p}(\omega^p)},
    $$
where 
$$C(\omega) = [\omega^{\frac{q_0 q}{q_0-q}}]_{A_{1+\frac{q_0 q}{q_0-q}\frac{p-p_0}{p_0 p}}}^{\frac{q_0-q}{q_0 q}(1+\max\{{\frac{p_0 p}{p - p_0} \cdot \frac{1}{q}, \frac{q_0 q}{q_0-q}  \cdot \frac{1}{p^{\prime}}}\}+\max\{p_0 p, \frac{p_0^2 p}{ p-p_0 }, q_0^{\prime} q^{\prime}, \frac{(q_0^{\prime})^2 q^{\prime}}{q^{\prime}-q_0^{\prime}}\} \cdot m)}.$$
\end{cor}
\begin{rem}
As we mentioned in section \ref{some application}, if $T$ is the fractional integral operator $I_{\alpha}$ defined by (\ref{eq-fractional integral}), then $T$ satisfies the assumptions of Theorem \ref{bloom estimate} with $p_0=1, q_0=\infty$. In this case, Theorem \ref{bloom estimate} provides a better estimate than \cite[Theorem 1.3]{AMR}. In the one-weight setting, we could obtain the sharp Bloom weighted estimate in \cite{CK,BMMS,AMR}, namely, for any $\omega \in A_{p,q}$,
$$
\|T_b^{m}f\|_{L^q(\omega^q)} \lesssim [\omega]_{A_{p,q}}^{(m+1-\frac{\alpha}{n})\max\{1, 
 \frac{{p^\prime}}{q}\}}\|b\|_{\text{BMO}} \|f\|_{L^{p}(\omega^p)}.
$$    
\end{rem}

Similar to what we mentioned earlier in  Corollary \ref{weighted estimate corollary}, 
we can obtain another Bloom weighted estimate under the assumption of $\omega \in A_{1+\frac{1}{p_0}-\frac{1}{p}} \cap \text{RH}_{q(\frac{q_0}{q})^{\prime}},$ that is,

\begin{cor}{\label{bloom estimate2}}
	Let $1 \leqslant p_0 < p < q < q_0 \leqslant \infty$ with  $\frac{1}{q} = \frac{1}{p} - \frac{\alpha}{n}$, $m \in \mathbb{N} \cup \{0\}$, $b \in L_{loc}^{1}(\mathbb{R}^n)$. Suppose $T$ is a sublinear operator and both $T$ and $\mathcal{M}_{T,q_0}^{\#}$ are locally weak $L^{p_0} \rightarrow L^{\frac{p_0n}{n-\alpha p_0}}$ bounded. Suppose $\omega$ is a weight satisfying $\omega \in A_{1+\frac{1}{p_0}-\frac{1}{p}} \cap \text{RH}_{q(\frac{q_0}{q})^{\prime}}$. Then, we have
    $$
\|T_b^{m}f\|_{L^q(\omega^q)} \lesssim C(\omega)\|b\|_{\text{BMO}} \|f\|_{L^{p}(\omega^p)},
    $$
where 
$$C(\omega) = ([\omega]_{A_{1+\frac{1}{p_0}-\frac{1}{p}}}[\omega]_{\text{RH}_{q(\frac{q_0}{q})^{\prime}}})^{1+\max\{{\frac{p_0 p}{p - p_0} \cdot \frac{1}{q}, \frac{q_0 q}{q_0-q}  \cdot \frac{1}{p^{\prime}}}\}+\max\{p_0 p, \frac{p_0^2 p}{ p-p_0 }, q_0^{\prime} q^{\prime}, \frac{(q_0^{\prime})^2 q^{\prime}}{q^{\prime}-q_0^{\prime}}\}  \cdot m}.$$
\end{cor}
\begin{rem}
As shown in the next section that the fractional power operators $L^{-\alpha / \kappa}$ satisfy the assumptions of Theorem \ref{bloom estimate2} with natural exponent $p_0$ and $q_0$. Then Theorem \ref{bloom estimate2} provides a quantitative Bloom weighted estimate for commutators of $L^{-\alpha / \kappa} $, which improve the results in \cite[Theorem 1.4]{AM1}.
\end{rem}

\begin{proof}[Proof of Theorem \ref{bloom estimate}]
From  \ref{Sparse Domination}, we obtain
$$\int_{\mathbb{R}^n}|T_{b}^{m}f||g| \ud x \lesssim \sum\limits_{j=1}^{3^{n}}B_{\mathcal{S}_j,b,p_0,q_0}^{m,\alpha}(f,g) + \sum\limits_{j=1}^{3^{n}}B_{\widetilde{\mathcal{S}}_j,b,q_0^{\prime},p_0^{\prime}}^{m,\alpha}(g,f),$$
where $\mathcal{S}_j \subset \mathscr{D}_j$ are sparse families and the sparse form $B_{\mathcal{S},b,p_0,q_0}^{m,\alpha}(f,g)$ is defined by
$$B_{\mathcal{S},b,p_0,q_0}^{m,\alpha}(f,g):=\sum\limits_{Q \in \mathcal{S}}\ave{|b-\ave{b}_Q|^{m}|f|}_{p_0,Q}\ave{|g|}_{q_0^{\prime},Q}|Q|^{1+\frac{\alpha}{n}}.$$
Thus, it suffices to prove that for any sparse family $\mathcal{S}$,
\begin{equation}{\label{sparse form}}
B_{\mathcal{S},b,p_0,q_0}^{m,\alpha}(f,g) \lesssim C_1(\mu,\lambda)\|b\|_{\text{BMO}_{\nu}}\|f\|_{L^{p}(\mu^{p})}\|g\|_{L^{q^{\prime}}(\lambda^{-q^{\prime}})},    
\end{equation}
where $C_1(\mu,\lambda)$ is given by (\ref{C1}). The proof of (\ref{sparse form}) is adapted from that of \cite[Theorem 5.1]{LLO}, where they obtained a Bloom weighted estimate for the sparse form with $\alpha=0$. In Lerner-Lorist-Ombrosi's proof, they obtained that there exists a sparse family $\mathcal{S}^{\prime} \subset \mathscr{D}$, such that for any $Q \in \mathcal{S}$, we have
$$\ave{|b-\ave{b}_Q|^{m}|f|}_{p_0,Q} \lesssim \|b\|_{\text{BMO}_{\nu}}^{m}\ave{(\mathcal{A}_{\mathcal{S}^{\prime},\nu}^{\floor{p_0 m}-1}(h))^{1/p_0}}_{p_0,Q},$$
where $\mathcal{A}_{\mathcal{S}^{\prime},\nu}^{m-1}(\cdot)$ is the $(\floor{p_0 m}-1)$-th iteration of the sparse operator $\mathcal{A}_{\mathcal{S}^{\prime},\nu}(\cdot)$ defined by
$$\mathcal{A}_{\mathcal{S}^{\prime},\nu}(\cdot):=\mathcal{A}_{\mathcal{S}^{\prime}}(\cdot)\nu, \qquad \mathcal{A}_{\mathcal{S}^{\prime}}(\cdot):=\sum\limits_{Q \in \mathcal{S}^{\prime}}\ave{\cdot}_Q \chi_Q,$$
and 
$$h:=\mathcal{A}_{\mathcal{S}^{\prime}}(|f|^{p_0})^{1-p_0m+\floor{p_0m}}\mathcal{A}_{\mathcal{S}^{\prime}}(\mathcal{A}_{\mathcal{S}^{\prime},\nu}(|f|^{p_0}))^{p_0m-\floor{p_0m}}\nu.$$
Then, applying Theorem \ref{weighted estimate} to $B_{\mathcal{S},b,p_0,q_0}^{m,\alpha}(f,g)$, we have
\begin{align*}
\mathcal{B}_{\mathcal{S},b,p_0,q_0}^{m,\alpha}(f,g) &\lesssim \|b\|_{BMO_{\nu}}^{m}\sum\limits_{Q \in \mathcal{S}}\ave{(\mathcal{A}_{\mathcal{S}^{\prime},\nu}^{\floor{p_0 m}-1}(h))^{1/p_0}}_{{p_0},Q}\ave{|g|}_{q_0^{\prime},Q}{|Q|}^{1+\frac{\alpha}{n}} \nonumber \\
& \lesssim \widetilde{C}(\lambda)\|b\|_{BMO_{\nu}}^{m} \|(\mathcal{A}_{\mathcal{S}^{\prime},\nu}^{\floor{p_0 m}-1}(h))^{1/p_0}\|_{L^{p}(\lambda^{p})}\|g\|_{L^{q^{\prime}}(\lambda^{-q^{\prime}})}, \nonumber 
\end{align*} 
where 
\begin{align*}
 \widetilde{C}(\lambda) = 
 & [\lambda^{\frac{p_0p}{p-p_0}},\lambda^{\frac{q_0q}{q_0-q}}]_{A_{\frac{1}{p_0}-\frac{1}{p},\frac{1}{q}-\frac{1}{q_0}}^{\frac{\alpha}{n}-\frac{1}{p_0}+\frac{1}{q_0}}}^{1+\max\{{\frac{p_0 p}{p - p_0} \cdot \frac{1}{q}, \frac{q_0 q}{q_0-q}  \cdot \frac{1}{p^{\prime}}}\}} \nonumber \\
= & [\lambda^{\frac{q_0 q}{q_0-q}}]_{A_{1+\frac{q_0 q}{q_0-q}\frac{p-p_0}{p_0 p}}}^{\frac{q_0-q}{q_0 q}(1+\max\{{\frac{p_0 p}{p - p_0} \cdot \frac{1}{q}, \frac{q_0 q}{q_0-q}  \cdot \frac{1}{p^{\prime}}}\})}. \nonumber 
\end{align*} 
As a result, we just need to prove

\begin{align}
  \|(\mathcal{A}_{\mathcal{S}^{\prime},\nu}^{\floor{p_0 m}-1}(h))^{1/p_0}\|_{L^{p}(\lambda^{p})} \lesssim & \,[\mu^{p}]_{A_{p/p_0}}^{\max\{1,\frac{p_0}{p-p_0}\}(p_0m -\floor{p_0m}+\frac{\floor{p_0m}^2}{2p_0 m}+\frac{\floor{p_0m}}{2p_0 m})} \label{bloom estimate2}\\  & \cdot [\lambda^{p}]_{A_{p/p_0}}^{\max\{1,\frac{p_0}{p-p_0}\}(\floor{p_0m}-\frac{\floor{p_0m}^2}{2p_0 m}-\frac{\floor{p_0m}}{2p_0 m})}\|f\|_{L^{p}(\mu^{p})} .\nonumber 
\end{align} 

Indeed, estimate (\ref{bloom estimate2}) can be obtained through a strategy similar to that of a \cite[Theorem 5.1]{LLO} by replacing all $u_j$ defined in their proof by $\frac{p}{p_0}$. 
\end{proof}

\section{Applications to fractional powers of differential operators}\label{App sec}

This section is dedicated in proving Theorem \ref{weighted estimate for L} and Theorem \ref{Bloom weighted estimate for L}. 

In this section, unless otherwise specified, we always assume that Let $L$ satisfy (A1) and (A2) for some $1\le p_0<q_0\le \infty$, $\epsilon>0$ and $\kappa>0$. We now define
\begin{equation}
	\label{eq-M Lalpha}
	\mathcal{M}_{L^{-\alpha/\kappa},q_0}f(x) = \sup_{Q\ni x}\Big[\fint_Q |L^{-\alpha/\kappa}(f\chi_{\mathbb R^n\backslash 3Q})|^{q_0}\ud x\Big]^{1/q_0},
\end{equation}
where the supremum is taken over all cubes $Q\subset \mathbb R^n$ containing $x$. It is easy to see from the definitions that $M^{\#}_{L^{-\alpha/\kappa},q_0}f$ defined by (\ref{sharp grand maximal function}) is controlled by $\mathcal{M}_{L^{-\alpha/\kappa},q_0}f$. 

Let $\alpha \in (0,n)$ and $1\le p<\infty$. For $f\in L^1_{\rm loc}(\mathbb R^n)$, we define the following maximal functions
\[
\mathcal M_p f(x) =\sup_{Q\ni x} \Big(\fint_Q |f(y)|^p \ud y\Big)^{1/p}, \ \ \ x\in \mathbb R^n
\]
and
\[
\mathcal M_{\alpha,p} f(x) =\sup_{Q\ni x} \ell(Q)^\alpha\Big(\fint_Q |f(y)|^p \ud y\Big)^{1/p}, \ \ \ x\in \mathbb R^n,
\]
where the supremum is taken over all cubes $Q$ containing $x$. 

By Corollary \ref{two-weight estimate} and Corollary \ref{bloom estimate2}, it suffices to prove:
\begin{thm}\label{thm- weak type for truncated operator}
	Let $0<\alpha<n(\f{1}{p_0}-\f{1}{q_0})$. Then the operator $L^{-\alpha/\kappa}$ and the truncation operator $\mathcal{M}_{L^{-\alpha/\kappa},q_0}$ are bounded from $L^{p_0}(\mathbb R^n)$ to $L^{p_0(\alpha),\infty}(\mathbb R^n)$, where $p_0(\alpha)=\f{p_0n}{n-\alpha p_0}$.  
\end{thm}

The weak type estimate of $L^{-\alpha/\kappa}$  is verified by the following proposition.
\begin{prop}\label{prop 1}
	Let $0<\alpha<n(\f{1}{p_0}-\f{1}{q_0})$. Then the fractional integral $L^{-\alpha/\kappa}$ 
 is bounded from $L^{p_0}(\mathbb R^n)$ to $L^{p_0(\alpha),\infty}(\mathbb R^n)$, where $p_0(\alpha)=\f{p_0n}{n-\alpha p_0}$.
\end{prop}
\begin{proof}
		
	Since $p_0<p_0(\alpha) <q_0$, we can choose $q_1, q_2$ such that  $p_0<q_1<p_0(\alpha)<q_2<q_0$. Then we write
	$$
	\begin{aligned}
		L^{-\alpha/\kappa}f &= c\int_0^\infty t^{\alpha/\kappa} e^{-tL}f\frac{\ud t}{t}\\
		&=c\int_0^A t^{\alpha/\kappa} e^{-tL}f\frac{\ud t}{t}+c\int_A^\infty t^{\alpha/\kappa} e^{-tL}f\frac{\ud t}{t}\\
		&=:f_A+f^A,
	\end{aligned}
	$$
	where $A$ will be fixed later. 
	
	This, together with Chebyshev's inequality and Minkowski's inequality, implies
	\[
	\begin{aligned}
		|\{x: &|L^{-\alpha/\kappa}f(x)|>\lambda \}|\\
        &\le |\{x: |f_A(x)|>\lambda/2 \}|+|\{x: |f^A(x)|>\lambda/2 \}|\\
		& \lesssim \lambda^{-q_1}\|f_A\|^{q_1}_{q_1} + \lambda^{-q_2}\|f^A\|^{q_2}_{q_2}\\
		&\lesssim \lambda^{-q_1}\Big[\int_0^A t^{\alpha/\kappa} \|e^{-tL}f\|_{q_1}\frac{\ud t}{t}\Big]^{q_1} + \lambda^{-q_2}\Big[\int_A^\infty t^{\alpha/\kappa} \|e^{-tL}f\|_{q_2}\frac{\ud t}{t}\Big]^{q_2}.
	\end{aligned} 	
	\]
	Using \eqref{Lpq estimate},
	\[
	\begin{aligned}
		|\{x: &| L^{-\alpha/\kappa}f (x)|>\lambda \}|\\
		&\lesssim \lambda^{-q_1}\|f\|_{p_0}^{q_1}\Big[\int_0^A t^{\alpha/\kappa} t^{\frac{n}{\kappa}(\frac{1}{q_1}-\frac{1}{p_0})}\frac{\ud t}{t}\Big]^{q_1} + \lambda^{-q_2}\|f\|_{p_0}^{q_2}\Big[\int_A^\infty t^{\alpha/\kappa} t^{\frac{n}{\kappa}(\frac{1}{q_2}-\frac{1}{p_0})}\frac{\ud t}{t}\Big]^{q_2}\\
		&\lesssim \lambda^{-q_1}\|f\|_{p_0}^{q_1}\Big[\int_0^A   t^{\frac{n}{\kappa}(\frac{1}{q_1}-\frac{1}{p_0(\alpha)})}\frac{\ud t}{t}\Big]^{q_1} + \lambda^{-q_2}\|f\|_{p_0}^{q_2}\Big[\int_A^\infty   t^{\frac{n}{\kappa}(\frac{1}{q_2}-\frac{1}{p_0(\alpha)})}\frac{\ud t}{t}\Big]^{q_2}\\
		&\lesssim \lambda^{-q_1}\|f\|_{p_0}^{q_1}A^{\frac{n}{\kappa}(1-\frac{q_1}{p_0(\alpha)})} +\lambda^{-q_2}\|f\|_{p_0}^{q_2}A^{\frac{n}{\kappa}(1-\frac{q_2}{p_0(\alpha)})}.
	\end{aligned} 	
	\]
	Taking $A= \lambda^{-\frac{\kappa p_0(\alpha)}{n}}\|f\|_{p_0}^{\frac{\kappa 
 p_0(\alpha)}{n}}$,
	\[
	|\{x: | L^{-\alpha/\kappa}f (x)|>\lambda \}|\lesssim \lambda^{-p_0(\alpha)}\|f\|_{p_0}^{p_0(\alpha)}.
	\]
	
	This completes our proof.
\end{proof}

The proof of the weak type estimate for the truncated operator $\mathcal M_{L^{-\alpha/\kappa}, q_0}$ is quite long and relies on the following technical results.
\begin{lem}
	\label{lem1} Let $0<\alpha<n(\f{1}{p_0}-\f{1}{q_0})$. For $N\in \mathbb N\cup\{0\}$, we have
	\[
	\Big[\fint_Q |L^{-\alpha/\kappa}(\ell(Q)^\kappa L)^Ne^{-\ell(Q)^\kappa L} (f\chi_{3Q})|^{q_0}\ud x\Big]^{1/q_0}\lesi \inf_{x\in Q}\mathcal M_{\alpha,p_0}f(x)
	\]	
	for all cubes $Q$. 
\end{lem}
\begin{proof}
	Using the formula
	$$
	L^{-\alpha/\kappa} = c\int_0^\infty s^{\alpha/\kappa} e^{-sL}\frac{\ud s}{s},
	$$
	we can write, for each $N\in \mathbb{N}\cup\{0\}$ and $t>0$,
	\begin{equation*}
		\begin{aligned}
			L^{-\alpha/\kappa}(\ell(Q)^\kappa L)^Ne^{-\ell(Q)^\kappa L}&=\int_0^{\infty}  s^{\alpha/\kappa}e^{-sL}(\ell(Q)^\kappa L)^Ne^{-\ell(Q)^\kappa L} \frac{\ud s}{s}\\
			&=\int_0^{\infty} \frac{s^{\alpha/\kappa}\ell(Q)^{\kappa N}}{(s+\ell(Q)^\kappa )^{N}}[(s+\ell(Q)^\kappa )L]^{N}e^{-(s+\ell(Q)^\kappa )L}\frac{\ud s}{s}.
		\end{aligned}
	\end{equation*}
	Applying \eqref{Lpq estimate} to obtain
	\[
	\begin{aligned}
		\|L^{-\alpha/\kappa}(\ell(Q)^\kappa L)^N&e^{-\ell(Q)^\kappa L}\|_{L^{p_0}(3Q)\to L^{q_0}(Q)}\\
		&\le \int_0^{\infty} \frac{s^{\alpha/\kappa}\ell(Q)^{\kappa N}}{(s+\ell(Q)^\kappa )^{N}}\left\|[(s+\ell(Q)^\kappa )L]^{N}e^{-(s+\ell(Q)^\kappa )L}\right\|_{p_0\to q_0} \frac{\ud s}{s}\\
		&\lesssim \int_0^{\infty} \frac{s^{\alpha/\kappa}\ell(Q)^{\kappa N}}{(s+\ell(Q)^\kappa )^{N}}(s+\ell(Q)^\kappa )^{-\f{n}{\kappa}(\f{1}{p_0}-\f{1}{q_0})} \frac{\ud s}{s}\\
		&\lesssim \int_0^{\ell(Q)^\kappa }\ldots + \int_{\ell(Q)^\kappa }^\infty\ldots\\
		&=:I_1 + I_2. 
	\end{aligned}
	\]
	For the term $I_1$, it is easy to see that $s+\ell(Q)^\kappa  \simeq \ell(Q)^\kappa $ in this situation. Hence,
	\[
	\begin{aligned}
		I_1 &\lesi \int_0^{\ell(Q)^\kappa } s^{\alpha/\kappa} \ell(Q)^{-n(\f{1}{p_0}-\f{1}{q_0})} \f{\ud s}{s}\\
		&\lesi \ell(Q)^{\alpha} \ell(Q)^{-n(\f{1}{p_0}-\f{1}{q_0})}\\
		&\simeq |Q|^{\alpha/n}|Q|^{\f{1}{q_0}-\f{1}{p_0}}.
	\end{aligned}
	\]
	For the second term $I_2$, we have $s+\ell(Q)^\kappa \simeq s$ for $s\ge \ell(Q)^\kappa $. In addition, $\alpha =n(\f{1}{p_0}-\f{1}{p_0(\alpha)})<n(\f{1}{p_0}-\f{1}{q_0})$. Consequently,
	\[
	\begin{aligned}
		I_2 &\lesi \int_{\ell(Q)^\kappa }^\infty \f{s^{\alpha/\kappa}\ell(Q)^{\kappa N}}{s^N  } s^{-\f{n}{\kappa}(\f{1}{p_0}-\f{1}{q_0})} \f{\ud s}{s}\\
		&\lesi \ell(Q)^{\alpha} \ell(Q)^{-n(\f{1}{p_0}-\f{1}{q_0})}\\
		&\simeq |Q|^{\alpha/n}|Q|^{\f{1}{q_0}-\f{1}{p_0}}.
	\end{aligned}
	\]
	Therefore,
	\[
	\|L^{-\alpha/\kappa}(\ell(Q)^\kappa L)^Ne^{-\ell(Q)^\kappa L}\|_{L^{p_0}(3Q)\to L^{q_0}(Q)}\lesi |Q|^{\alpha/n}|Q|^{\f{1}{q_0}-\f{1}{p_0}},
	\]
	which implies
	\[
	\begin{aligned}
		\Big[\fint_Q |L^{-\alpha/\kappa}(\ell(Q)^\kappa L)^Ne^{-\ell(Q)^\kappa L} (f\chi_{3Q})|^{q_0}\ud x\Big]^{1/q_0}&\lesi |Q|^{\alpha/n}\Big[\fint_{3Q} |f|^{p_0}\ud x\Big]^{1/p_0}\\
		&\lesi \inf_{x\in Q}\mathcal M_{\alpha,p_0}f(x).
	\end{aligned}
	\]
	
	This completes our proof.
\end{proof} 

For every $t>0$ and $N\in \mathbb N$ we define
\begin{equation}
	\label{defn Q and P}
	Q_{N,t}(L) = c_N (tL)^Ne^{-tL}, \ \ \ \text{and} \ \ \ P_{N,t}(L) = \int_1^\infty Q_{N,st}(L)\frac{\ud s}{s}=\int_t^\infty Q_{N,s}(L)\frac{\ud s}{s}
\end{equation}
where $\displaystyle c_N = \Big(\int_0^\infty s^N e^{-s}\frac{\ud s}{s}\Big)^{-1}=\frac{1}{(N-1)!}$.

By  using integration by parts, we can see that $P_{N,t}(L)=p(tL)e^{-tL}$ where $p$ is a polynomial of degree $N-1$ and $p(0)=1$. Therefore,  $P_{N,t}(L)$ satisfies (A1) and (A2).

\begin{lem}
	\label{lem2} Let $0<\alpha<n(\f{1}{p_0}-\f{1}{q_0})$. For $N\in \mathbb N$ with $N >\kappa(\alpha+n/q_0+\epsilon)$,
	\[
	\Big[\fint_Q |L^{-\alpha/\kappa}(I-P_{N, \ell(Q)^\kappa }(L))(f\chi_{\mathbb R^n\backslash 3Q})|^{q_0}\ud x\Big]^{1/q_0}\lesi \inf_{x\in Q}\mathcal M_{\alpha,p_0}f(x)
	\] 
	for all cubes $Q$.
\end{lem}
\begin{proof}
	Note that
	\[
	I-P_{N, \ell(Q)^\kappa }(L)=\int_0^{\ell(Q)^\kappa } Q_{N,t}(L)\frac{\ud t}{t}
	\]
	and
	\[
	L^{-\alpha/\kappa} = c\int_0^\infty s^{\alpha/\kappa} e^{-sL}\frac{\ud s}{s}.
	\]
	Hence,
	\begin{equation}
		\label{eq-I3 1}
		\begin{aligned}
			\Big[\fint_Q |L^{-\alpha/\kappa}&(I-P_{N, \ell(Q)^\kappa }(L))(f\chi_{\mathbb R^n\backslash 3Q})|^{q_0}\ud x\Big]^{1/q_0}\\
			&\le \Big[\fint_Q \Big|\Big(\int_0^\infty \int_0^{\ell(Q)^\kappa } s^{\alpha/\kappa}e^{-sL}Q_{N,t}(L)(f\chi_{\mathbb R^n\backslash 3Q})\frac{\ud t}{t}\f{\ud s}{s}\Big)\Big|^{q_0}\ud x\Big]^{1/q_0}\\
			&= \Big[\fint_Q \Big|\Big(\int_0^\infty \int_0^{\ell(Q)^\kappa } \f{s^{\alpha/\kappa}t^N}{(s+t)^N} Q_{N,s+t}(L)(f\chi_{\mathbb R^n\backslash 3Q})\frac{\ud t}{t}\f{\ud s}{s}\Big)\Big|^{q_0}\ud x\Big]^{1/q_0}\\
			&\le \sum_{j\ge 2} \Big[\fint_Q \Big|\Big(\int_0^\infty \int_0^{\ell(Q)^\kappa } \f{s^{\alpha/\kappa}t^N}{(s+t)^N} Q_{N,s+t}(L)f_j\frac{\ud t}{t}\f{\ud s}{s}\Big)\Big|^{q_0}\ud x\Big]^{1/q_0},
		\end{aligned}
	\end{equation}
	where $f_j= f\chi_{S_j(Q)}$ for $j\ge 2$. See (A2) for the definition of $S_j(Q)$.
	
	For each $j$, by Minkowski's inequality,
	\[
	\begin{aligned}
		\Big[\fint_Q &\Big|\Big(\int_0^\infty \int_0^{\ell(Q)^\kappa } \f{s^{\alpha/\kappa}t^N}{(s+t)^N} Q_{N,s+t}(L)f_j\frac{\ud t}{t}\f{\ud s}{s}\Big)\Big|^{q_0}\ud x\Big]^{1/q_0}\\
		&\lesi  \int_0^\infty \int_0^{\ell(Q)^\kappa } \f{s^{\alpha/\kappa}t^N}{(s+t)^N} \|Q_{N,s+t}(L)f_j\|_{(L^{q_0}(Q), \f{\ud x}{|Q|})}\frac{\ud t}{t}\f{\ud s}{s}\\
		&\lesi  \int_0^{\ell(Q)^\kappa }\int_0^{\ell(Q)^\kappa }\ldots + \int_{\ell(Q)^\kappa }^\infty\int_0^{\ell(Q)^\kappa }\ldots =: E_1 + E_2. 
	\end{aligned}
	\]
	We now take care of the first term $E_1$. By \eqref{eq2-Tt}, for $\ell(Q)^\kappa \ge \max\{s,t\}$,
	\[
	\begin{aligned}
		\Big(\fint_{Q}&|Q_{N,s+t}(L)f_j|^{q_0}\Big)^{1/q_0}\\
		&\lesi 
		\max\Big\{\Big(\f{3^j\ell(Q)}{(s+t)^{1/\kappa}}\Big)^n,\Big(\f{3^j\ell(Q)}{(s+t)^{1/\kappa}}\Big)^{n/p_0} \Big\} \Big(1+\f{(s+t)^{1/\kappa}}{\ell(Q)}\Big)^{n/q_0}\Big(1+\f{3^j\ell(Q)}{(s+t)^{1/\kappa}}\Big)^{-n-\epsilon}\\
		& \ \ \ \ \ \ \ \hskip2cm \times \Big(\fint_{S_j(Q)}|f|^{p_0}\Big)^{1/p_0}\\
		&\lesi 
		\Big(\f{3^j\ell(Q)}{(s+t)^{1/\kappa}}\Big)^n  \Big(\f{3^j\ell(Q)}{(s+t)^{1/\kappa}}\Big)^{-n-\epsilon} \Big(\fint_{S_j(Q)}|f|^{p_0}\Big)^{1/p_0}\\
		&\lesi \Big(\f{(s+t)^{1/\kappa}}{3^j\ell(Q)}\Big)^{\epsilon} \Big(\fint_{S_j(Q)}|f|^{p_0}\Big)^{1/p_0}.
	\end{aligned}
	\]
	Hence,
	\[
	\begin{aligned}
		E_1&\lesi \int_0^{\ell(Q)^\kappa } \int_0^{\ell(Q)^\kappa } \f{s^{\alpha/\kappa}t^N}{(s+t)^N} \Big(\f{(s+t)^{1/\kappa}}{3^j\ell(Q)}\Big)^{\epsilon}  \frac{\ud t}{t}\f{\ud s}{s} \ \times \Big(\fint_{S_j(Q)}|f|^{p_0}\Big)^{1/p_0}.
	\end{aligned}
	\]
	On the other hand,
	\[
	\begin{aligned}
		\int_0^{\ell(Q)^\kappa } &\int_0^{\ell(Q)^\kappa } \f{s^{\alpha/\kappa}t^N}{(s+t)^N} \Big(\f{(s+t)^{1/\kappa}}{3^j\ell(Q)}\Big)^{\epsilon}  \frac{\ud t}{t}\f{\ud s}{s}\\
		&\lesi \int_0^{\ell(Q)^\kappa } \int_0^{s} \f{s^{\alpha/\kappa}t^N}{(s+t)^N} \Big(\f{(s+t)^{1/\kappa}}{3^j\ell(Q)}\Big)^{\epsilon}  \frac{\ud t}{t}\f{\ud s}{s} \\
        & \qquad + \int_0^{\ell(Q)^\kappa } \int_s^{\ell(Q)^\kappa } \f{s^{\alpha/\kappa}t^N}{(s+t)^N} \Big(\f{(s+t)^{1/\kappa}}{3^j\ell(Q)}\Big)^{\epsilon}  \frac{\ud t}{t}\f{\ud s}{s}\\
		&\simeq \int_0^{\ell(Q)^\kappa } \int_0^{s} \f{s^{\alpha/\kappa}t^N}{s^N} \Big(\f{s^{1/\kappa}}{3^j\ell(Q)}\Big)^{\epsilon}  \frac{\ud t}{t}\f{\ud s}{s}+ \int_0^{\ell(Q)^\kappa } \int_s^{\ell(Q)^\kappa } \f{s^{\alpha/\kappa}t^N}{t^N} \Big(\f{t^{1/\kappa}}{3^j\ell(Q)}\Big)^{\epsilon}  \frac{\ud t}{t}\f{\ud s}{s}\\
		&\lesi 3^{-j\epsilon}\ell(Q)^\alpha. 
	\end{aligned}
	\]
	It follows that
	\[
	\begin{aligned}
		E_1&\lesi    2^{-j\epsilon}\ell(Q)^\alpha\Big(\fint_{S_j(Q)}|f|^{p_0}\Big)^{1/p_0}\\
		&\lesi 2^{-j\epsilon}\inf_{x\in Q} \mathcal M_{\alpha,p_0}f(x).
	\end{aligned}
	\]
	For the second term $E_2$, by \eqref{eq2-Tt}, for $t\le \ell(Q)^\kappa \le  s$,
	\[
	\begin{aligned}
		\Big(\fint_{Q}&|Q_{N,s+t}(L)f_j|^{q_0}\Big)^{1/q_0}\\
		&\lesi 
		\max\Big\{\Big(\f{3^j\ell(Q)}{(s+t)^{1/\kappa}}\Big)^n,\Big(\f{3^j\ell(Q)}{(s+t)^{1/\kappa}}\Big)^{n/p_0} \Big\} \Big(1+\f{(s+t)^{1/\kappa}}{\ell(Q)}\Big)^{n/q_0}\Big(1+\f{3^j\ell(Q)}{(s+t)^{1/\kappa}}\Big)^{-n-\epsilon}\\
		& \ \ \ \ \ \ \ \hskip2cm \times \Big(\fint_{S_j(Q)}|f|^{p_0}\Big)^{1/p_0}\\
		&\lesi 
		\max\Big\{\Big(\f{3^j\ell(Q)}{s^{1/\kappa}}\Big)^n,\Big(\f{3^j\ell(Q)}{s^{1/\kappa}}\Big)^{n/p_0} \Big\} \Big(\f{s^{1/\kappa}}{\ell(Q)}\Big)^{n/q_0}\Big(1+\f{3^j\ell(Q)}{s^{1/\kappa}}\Big)^{-n-\epsilon} \Big(\fint_{S_j(Q)}|f|^{p_0}\Big)^{1/p_0}.
	\end{aligned}
	\]
	Hence,
	\[
	\begin{aligned}
		E_2&\lesi \int_{\ell(Q)^\kappa }^{[3^j\ell(Q)]^\kappa} \int_0^{\ell(Q)^\kappa } \f{s^{\alpha/\kappa}t^N}{s^N} \Big(\f{s^{1/\kappa}}{\ell(Q)}\Big)^{n/q_0}\Big(\f{3^j\ell(Q)}{s^{1/\kappa}}\Big)^{-\epsilon} \frac{\ud t}{t}\f{\ud s}{s} \times \Big(\fint_{S_j(Q)}|f|^{p_0}\Big)^{1/p_0}\\
		&\ \ + \int_{[3^j\ell(Q)]^\kappa}^\infty \int_0^{\ell(Q)^\kappa } \f{s^{\alpha/\kappa}t^N}{s^N}  \Big(\f{3^j\ell(Q)}{s^{1/\kappa}}\Big)^{n/p_0}  \Big(\f{s^{1/\kappa}}{\ell(Q)}\Big)^{n/q_0}  \frac{\ud t}{t}\f{\ud s}{s} \times \Big(\fint_{S_j(Q)}|f|^{p_0}\Big)^{1/p_0}\\
		&\lesi \int_{\ell(Q)^\kappa }^{[3^j\ell(Q)]^\kappa}   \f{s^{\alpha/\kappa}\ell(Q)^{\kappa N}}{s^N}     \Big(\f{s^{1/\kappa}}{\ell(Q)}\Big)^{n/q_0}\Big(\f{s^{1/\kappa}}{3^j\ell(Q)}\Big)^{\epsilon}  \f{\ud s}{s} \times \Big(\fint_{S_j(Q)}|f|^{p_0}\Big)^{1/p_0}\\
		&\ \ + \int_{[3^j\ell(Q)]^\kappa}^\infty   \f{s^{\alpha/\kappa}\ell(Q)^{\kappa N}}{s^N}  \Big(\f{3^j\ell(Q)}{s^{1/\kappa}}\Big)^{n/p_0}  \Big(\f{s^{1/\kappa}}{\ell(Q)}\Big)^{n/q_0}   \f{\ud s}{s} \times \Big(\fint_{S_j(Q)}|f|^{p_0}\Big)^{1/p_0}\\
		&\lesi    3^{-j\epsilon}\ell(Q)^\alpha\Big(\fint_{S_j(Q)}|f|^{p_0}\Big)^{1/p_0}\\
		&\lesi 3^{-j\epsilon}\inf_{x\in Q} \mathcal M_{\alpha,p_0}f(x),
	\end{aligned}
	\]
	as long as $N>\kappa(\alpha+n/q_0+\epsilon)$.
	
	Collecting the estimates of $E_1$ and $E_2$, we have
	\[
	\Big[\fint_Q \Big|\Big(\int_0^\infty \int_0^{\ell(Q)^\kappa } \f{s^{\alpha/\kappa}t^N}{(s+t)^N} Q_{N,s+t}(L)f_j\frac{\ud t}{t}\f{\ud s}{s}\Big)\Big|^{q_0}\ud x\Big]^{1/q_0}\lesi 3^{-j\epsilon}\inf_{x\in Q} \mathcal M_{\alpha,p_0}f(x). 
	\]
	This, together with \eqref{eq-I3 1}, implies that
	$$
	\Big[\fint_Q |L^{-\alpha/\kappa}(I-P_{N, \ell(Q)^\kappa }(L))(f\chi_{\mathbb R^n\backslash 3Q})|^{q_0}\ud x\Big]^{1/q_0}\lesi \inf_{x\in Q} \mathcal M_{\alpha,p_0}f(x).
	$$
	This completes our proof.
\end{proof}

Let $P_{N,t}(L)$ be as in \eqref{defn Q and P}. We now define
\begin{equation}\label{eq- Tsharp L}
	\begin{aligned}
		T^\sharp_{L} f(x) 
		&= \sup_{Q\ni x} \Big[\fint_Q |P_{N, \ell(Q)^\kappa }(L)L^{-\alpha/\kappa}f(y)|^{q_0}\ud y\Big]^{1/q_0}
	\end{aligned}
\end{equation}
and the supremum is taken over all the cubes $Q$ containing $x$. Then we have:

\begin{lem}\label{lem3} Let $0<\alpha<n(\f{1}{p_0}-\f{1}{q_0})$. For $N\in \mathbb N$, the operator $ T^\sharp_{L}$ defined by \eqref{eq- Tsharp L} is bounded from $L^{p_0}(\mathbb R^n)$ to $L^{p_0(\alpha),\infty}(\mathbb R^n)$, where $p_0(\alpha)=\f{p_0n}{n-\alpha p_0}$.
	
\end{lem}
\begin{proof}
	
	Since $P_{N,t}(L)=p(tL)e^{-tL}$ where $p$ is a polynomial of degree $N-1$ and $p(0)=1$, $P_{N,t}(L)$ satisfies (A1) and (A2). Consequently, for $x\in \mathbb R^n$,
	\[
	T_L^\sharp f(x) \lesssim \mathcal M_{p_0}(L^{-\alpha/\kappa}f)(x).
	\]
	Therefore, it suffices to show that 
	\[
	|\{x: |\mathcal M_{p_0}(L^{-\alpha/\kappa}f)(x)|>\lambda \}|\lesssim \lambda^{-p_0(\alpha)}\|f\|^{p_0(\alpha)}_{p_0}
	\]
	for all $\lambda>0$, where $p_0(\alpha):= \frac{p_0n}{n-\alpha p_0}$.	
	
	This is similar to the proof of Proposition \ref{prop 1}. Indeed, since $p_0<p_0(\alpha) <q_0$, we can choose $q_1, q_2$ such that  $p_0<q_1<p_0(\alpha)<q_2<q_0$. Then we write
	$$
	\begin{aligned}
		L^{-\alpha/\kappa}f &= c\int_0^\infty t^{\alpha/\kappa} e^{-tL}f\frac{\ud t}{t}\\
		&=c\int_0^A t^{\alpha/\kappa} e^{-tL}f\frac{\ud t}{t}+c\int_A^\infty t^{\alpha/\kappa} e^{-tL}f\frac{\ud t}{t}\\
		&=:f_A+f^A,
	\end{aligned}
	$$
	where $A$ will be fixed later. 
	
	This, together with Chebyshev's inequality, implies
	\[
	\begin{aligned}
		|\{x: |\mathcal M_{p_0}(L^{-\alpha/\kappa}f)(x)|>\lambda \}|&\le |\{x: |\mathcal M_{p_0}f_A(x)|>\lambda/2 \}|+|\{x: |\mathcal M_{p_0}f^A(x)|>\lambda/2 \}|\\
		&\lesssim \lambda^{-q_1}\|\mathcal M_{p_0}f_A\|^{q_1}_{q_1} + \lambda^{-q_2}\|\mathcal M_{p_0}f^A\|^{q_2}_{q_2}.
	\end{aligned}
	\]
	Since $\mathcal M_{p_0}$ is bounded on $L^p, p>p_0$, we further obtain
	\[
	\begin{aligned}
		|\{x: &|\mathcal M_{p_0}(L^{-\alpha/\kappa}f)(x)|>\lambda \}| \\
		&\lesssim \lambda^{-q_1}\|f_A\|^{q_1}_{q_1} + \lambda^{-q_2}\|f^A\|^{q_2}_{q_2}\\
		&\lesssim \lambda^{-q_1}\Big[\int_0^A t^{\alpha/\kappa} \|e^{-tL}f\|_{q_1}\frac{\ud t}{t}\Big]^{q_1} + \lambda^{-q_2}\Big[\int_A^\infty t^{\alpha/\kappa} \|e^{-tL}f\|_{q_2}\frac{\ud t}{t}\Big]^{q_2}\\
	\end{aligned} 	
	\]
	Using \eqref{Lpq estimate},
	\[
	\begin{aligned}
		|\{x: &|\mathcal M_{p_0}(L^{-\alpha/\kappa}f)(x)|>\lambda \}|\\
		&\lesssim \lambda^{-q_1}\|f\|_{p_0}^{q_1}\Big[\int_0^A t^{\alpha/\kappa} t^{\frac{n}{\kappa}(\frac{1}{q_1}-\frac{1}{p_0})}\frac{\ud t}{t}\Big]^{q_1} + \lambda^{-q_2}\|f\|_{p_0}^{q_2}\Big[\int_A^\infty t^{\alpha/\kappa} t^{\frac{n}{\kappa}(\frac{1}{q_2}-\frac{1}{p_0})}\frac{\ud t}{t}\Big]^{q_2}\\
		&\lesssim \lambda^{-q_1}\|f\|_{p_0}^{q_1}\Big[\int_0^A   t^{\frac{n}{\kappa}(\frac{1}{q_1}-\frac{1}{p_0(\alpha)})}\frac{\ud t}{t}\Big]^{q_1} + \lambda^{-q_2}\|f\|_{p_0}^{q_2}\Big[\int_A^\infty   t^{\frac{n}{\kappa}(\frac{1}{q_2}-\frac{1}{p_0(\alpha)})}\frac{\ud t}{t}\Big]^{q_2}\\
		&\lesssim \lambda^{-q_1}\|f\|_{p_0}^{q_1}A^{\frac{n}{\kappa}(1-\frac{q_1}{p_0(\alpha)})} +\lambda^{-q_2}\|f\|_{p_0}^{q_2}A^{\frac{n}{\kappa}(1-\frac{q_2}{p_0(\alpha)})}.
	\end{aligned} 	
	\]
	Taking $A= \lambda^{-\frac{\kappa p_0(\alpha)}{n}}\|f\|_{p_0}^{\frac{\kappa p_0(\alpha)}{n}}$,
	\[
	|\{x: |\mathcal M_{p_0}(L^{-\alpha/\kappa}f)(x)|>\lambda \}|\lesssim \lambda^{-p_0(\alpha)}\|f\|_{p_0}^{p_0(\alpha)}.
	\]
	
	This completes our proof.
\end{proof}
We now ready to give the proof of Theorem \ref{thm- weak type for truncated operator}.
\begin{proof}[Proof of Theorem \ref{thm- weak type for truncated operator}:] We proved in Proposition \ref{prop 1} that  $L^{-\alpha/\kappa}$ is bounded from $L^{p_0}(\mathbb R^n)$ to $L^{p_0(\alpha),\infty}(\mathbb R^n)$. It remains to take care of  the operator $\mathcal{M}_{L^{-\alpha/\kappa},q_0}$.

To do this, we fix $N>\kappa(\alpha+n/q_0+\epsilon)$. For every $t>0$, recall from \eqref{defn Q and P} that  
	\begin{equation*}
		Q_{N,t}(L) = c_N (tL)^Ne^{-tL}, \ \ \ \text{and} \ \ \ P_{N,t}(L) = \int_1^\infty Q_{N,st}(L)\frac{\ud s}{s}=\int_t^\infty Q_{N,s}(L)\frac{\ud s}{s}
	\end{equation*}
	where $\displaystyle c_N = \Big(\int_0^\infty s^N e^{-s}\frac{\ud s}{s}\Big)^{-1}=\frac{1}{(N-1)!}$.
	
	Fix $x\in \mathbb R^n$ and let $Q$ be a cube containing $x$. Then, we have
	\begin{equation}
		\label{eq-MT main thm 1}
		\begin{aligned}
			\Big[\fint_Q &|L^{-\alpha/\kappa}(f\chi_{\mathbb R^n\backslash 3Q})(y)|^{q_0}\ud y\Big]^{1/q_0}\\
			&\lesssim \Big[\fint_Q |L^{-\alpha/\kappa}P_{N, \ell(Q)^\kappa }(L)(f\chi_{\mathbb R^n\backslash 3Q})(y)|^{q_0}\ud y\Big]^{1/q_0}\\
			& \ \ +\Big[\fint_Q |L^{-\alpha/\kappa}(I-P_{N, \ell(Q)^\kappa }(L))(f\chi_{\mathbb R^n\backslash 3Q})(y)|^{q_0}\ud y\Big]^{1/q_0}\\
			&\lesssim \Big[\fint_Q |L^{-\alpha/\kappa}P_{N, \ell(Q)^\kappa }(L)f(y)|^{q_0}\ud y\Big]^{1/q_0}+\Big[\fint_Q |L^{-\alpha/\kappa}P_{N, \ell(Q)^\kappa }(L)(f\chi_{3Q})(y)|^{q_0}\ud y\Big]^{1/q_0}\\
			&\ \ \ +\Big[\fint_Q |L^{-\alpha/\kappa}(I-P_{N, \ell(Q)^\kappa }(L))(f\chi_{\mathbb R^n\backslash 3Q})(y)|^{q_0}\ud y\Big]^{1/q_0}.
		\end{aligned}
	\end{equation}
	
	By Lemmas \ref{lem1}, \ref{lem2} and \eqref{eq- Tsharp L}, we have
	\[
	\sup_{Q\ni x}\Big[\fint_Q |L^{-\alpha/\kappa}(f\chi_{\mathbb R^n\backslash 3Q})(y)|^{q_0}\ud y\Big]^{1/q_0}\lesi \mathcal M_{\alpha,p_0}f(x) + T_L^\sharp f(x).
	\]
	The conclusion follows directly from the facts that $\mathcal M_{\alpha,p_0}$ is bounded from $L^{p_0}(\mathbb R^n)$ to $L^{p_0(\alpha),\infty}(\mathbb R^n)$ and $T^\sharp_L$ is bounded from $L^{p_0}(\mathbb R^n)$ to $L^{p_0(\alpha),\infty}(\mathbb R^n)$ (see Lemma \ref{lem3}.)

	This completes our proof.
\end{proof}

\bigskip

\textbf{Acknowledgment.} T. A. Bui was supported by the Australian Research Council via the grant ARC DP220100285. L. Zheng was supported by the National Key R\&D Program of China (Grant No. 2021YFA1002500). The authors would like to thank Kangwei Li for his useful discussion and suggestions.


\end{document}